\documentclass[12pt,oneside,english]{amsart}
\usepackage[T1]{fontenc}
\usepackage[utf8]{inputenc}
\usepackage{geometry}
\usepackage{fancyhdr}
\usepackage{color}
\usepackage{float}
\usepackage{amstext}
\usepackage{amsthm}
\usepackage{amssymb}
\usepackage{mathdots}
\usepackage{graphicx}
\usepackage{esint}
\PassOptionsToPackage{normalem}{ulem}
\usepackage{ulem}

\makeatletter

\providecommand{\tabularnewline}{\\}

\numberwithin{equation}{section}
\numberwithin{figure}{section}
\theoremstyle{plain}
\newtheorem{thm}{\protect\theoremname}
\theoremstyle{plain}
\newtheorem{prop}[thm]{\protect\propositionname}
\theoremstyle{remark}
\newtheorem*{rem*}{\protect\remarkname}
\theoremstyle{plain}
\newtheorem{lem}[thm]{\protect\lemmaname}

\usepackage[english]{babel}
\usepackage{fancyhdr}
\usepackage{tikz}
\usetikzlibrary{arrows,shapes}
\usetikzlibrary{calc,decorations.markings}
\tikzset{
  reversed with radius/.style={
    x radius=#1,
    y radius=-#1,
 }
}

\tikzset{
  with arrows/.style={
    decoration={ markings,
      mark=between positions #1 and .999 step #1 with {\arrow{stealth}}
    }, postaction={decorate}
  }, with arrows/.default=25mm,
} 

\providecommand{\tabularnewline}{\\}

\providecommand{\lemmaname}{Lemma}
\providecommand{\propositionname}{Proposition}
\providecommand{\theoremname}{Theorem}

\newcommand{\abs}[1]{\ensuremath{|#1|}}

\newcommand{\Abs}[1]{\ensuremath{\left|#1\right|}}

\newcommand{\norm}[2]{\ensuremath{|\!|#1|\!|_{#2}}}

\newcommand{\Norm}[2]{\ensuremath{\left|\!\left|#1\right|\!\right|_{#2}}}

\renewcommand{\d}[1]{\ensuremath{\textnormal{d}#1}}

\newcommand{\cC}{\mathcal{C}}

\newcommand{\cM}{\mathcal{M}}

\newcommand{\cO}{\mathcal{O}}

\newcommand{\cS}{\mathcal{S}}

\usepackage{babel}
\providecommand{\lemmaname}{Lemma}
  \providecommand{\propositionname}{Proposition}
  \providecommand{\remarkname}{Remark}
\providecommand{\theoremname}{Theorem}

\usepackage[section]{placeins}

\usepackage{babel}
\providecommand{\lemmaname}{Lemma}
  \providecommand{\propositionname}{Proposition}
\providecommand{\theoremname}{Theorem}

\usepackage{babel}

\providecommand{\lemmaname}{Lemma}
\providecommand{\propositionname}{Proposition}
\providecommand{\remarkname}{Remark}
\providecommand{\theoremname}{Theorem}

\usepackage{babel}

\providecommand{\lemmaname}{Lemma}
\providecommand{\propositionname}{Proposition}
\providecommand{\remarkname}{Remark}
\providecommand{\theoremname}{Theorem}

\usepackage{babel}
\providecommand{\lemmaname}{Lemma}
\providecommand{\propositionname}{Proposition}
\providecommand{\remarkname}{Remark}
\providecommand{\theoremname}{Theorem}

\usepackage{babel}
\providecommand{\lemmaname}{Lemma}
\providecommand{\propositionname}{Proposition}
\providecommand{\remarkname}{Remark}
\providecommand{\theoremname}{Theorem}

\makeatother

\usepackage{babel}
\providecommand{\lemmaname}{Lemma}
\providecommand{\propositionname}{Proposition}
\providecommand{\remarkname}{Remark}
\providecommand{\theoremname}{Theorem}

\begin{document}

\title{On the Fourier coefficients of powers of a Blaschke factor and strongly
annular fonctions}

\author{Alexander Borichev}

\address{Aix-Marseille University, CNRS, Centrale Marseille, I2M, Marseille,
France.}

\email{alexander.borichev@math.cnrs.fr}

\author{Karine Fouchet}

\address{Aix-Marseille University, CNRS, Centrale Marseille, I2M, Marseille,
France.}

\email{karine.isambard@univ-amu.fr }

\author{Rachid Zarouf}

\address{Aix-Marseille University, Laboratoire Apprentissage, Didactique,
Evaluation, Formation, Campus Universitaire de Saint-Jérôme, 52 Avenue
Escadrille Normandie Niemen, 13013 Marseille}

\email{rachid.zarouf@univ-amu.fr}

\address{Department of Mathematics and Mechanics, Saint Petersburg State University,
28, Universitetski pr., St. Petersburg, 198504, Russia.}

\email{rzarouf@gmail.com}



\thanks{The work is supported by the project ANR 18-CE40-0035}
\begin{abstract}
We compute asymptotic formulas for the $k^{{\rm th}}$ Fourier coefficients
of $b_{\lambda}^{n}$, where $b_{\lambda}(z)=\frac{z-\lambda}{1-\lambda z}$
is the Blaschke factor associated to $\lambda\in\mathbb{D}$, $k\in[0,\infty)$
and $n$ is a large integer. We distinguish several regions of different
asymptotic behavior of those coefficients in terms of $k$ and $n$.
Given $\beta\in((1-\lambda)/(1+\lambda),(1+\lambda)/(1-\lambda))$
their decay is oscillatory for $k\in[\beta n,n/\beta]$. Given $\alpha\in(0,(1-\lambda)/(1+\lambda))$
their decay is exponential for $k\in[0,n\alpha]\cup[n/\alpha,\infty).$
Airy-type behavior is happening near the $k$-transition points $n(1-\lambda)/(1+\lambda)$
and $n(1+\lambda)/(1-\lambda)$. The asymptotic formulas for the $k^{{\rm th}}$
Fourier coefficients of $b_{\lambda}^{n}$ are derived using standard
tools of asymptotic analysis of Laplace-type integrals. More precisely,
the integral defining the $k^{{\rm th}}$ Fourier coefficient of $b_{\lambda}^{n}$
is perfectly suited for an application of the method of stationary
phase when $k\in\left(n(1-\lambda)/(1+\lambda),n(1+\lambda)/(1-\lambda)\right)$
and requires the use of the method of the steepest descent when $k\notin[n(1-\lambda)/(1+\lambda),n(1+\lambda)/(1-\lambda)]$.
Uniform versions of those standard methods are required when $k$
approaches one of the boundaries $n(1-\lambda)/(1+\lambda),$ $n(1+\lambda)/(1-\lambda)$.
As an application, we construct strongly annular functions with Taylor
coefficients satisfying sharp summation properties. 
\end{abstract}

\maketitle

\specialsection{Introduction}

\subsection{Notation}

Let $\mathbb{D}=\{z:\:\:|z|<1\}$ be the open unit disk and $\partial\mathbb{D}$
its boundary. For a given $\lambda\in\mathbb{D}$ we denote by 
\[
b_{\lambda}(z)=\frac{z-\lambda}{1-\bar{\lambda}z},
\]
the Blaschke factor corresponding to $\lambda\in\mathbb{D}$. It is
well-known that the function $b_{\lambda}$ is an automorphism of
$\mathbb{D}$ and that $\abs{b_{\lambda}(z)}=1\iff z\in\partial\mathbb{D}$.
Given a nonnegative integer $n$ we recall the definition of the $k^{{\rm th}}-$Taylor/Fourier
coefficient of the $n^{{\rm th}}-$power of $b_{\lambda}$: 
\[
\frac{(b_{\lambda}^{n})^{(k)}(0)}{k!}=\widehat{b_{\lambda}^{n}}(k)=\frac{1}{2i\pi}\int_{\partial\mathbb{D}}b_{\lambda}^{n}(z)z^{-k}\frac{\d z}{z},
\]
whose asymptotic behavior we wish to determine as $n\rightarrow\infty$.
Let $b_{\lambda}^{n}(z)=\sum_{k\geq0}\widehat{b_{\lambda}^{n}}(k)z^{k}$
be the Taylor expansion of $b_{\lambda}^{n}$. Then we have 
\begin{align*}
\left(\frac{z-\lambda}{1-\bar{\lambda}z}\right)^{n} & =e^{in\theta}\left(\frac{ze^{-i\theta}-\abs{\lambda}}{1-\abs{\lambda}ze^{-i\theta}}\right)^{n}\\
 & =\sum_{k\geq0}\widehat{b_{\abs{\lambda}}^{n}}(k)e^{i(n-k)\theta}z^{k},
\end{align*}
with $\theta=\arg\lambda,$ which shows that $\widehat{b_{\lambda}^{n}}(k)=\widehat{b_{\abs{\lambda}}^{n}}(k)e^{i(n-k)\theta}$.
Therefore, without loss of generality we assume from now on that $\lambda\in(0,1)$.
In this article we compute asymptotic formulas for $\widehat{b_{\lambda}^{n}}(k)$
as $n\rightarrow\infty$, when $k\in[0,\infty)$. Furthermore, we
apply these asymptotic formulas to construct strongly annular functions
with small Taylor coefficients.

\subsection{Motivations}

Various motivations have led to study the asymptotic behavior of $\widehat{b_{\lambda}^{n}}(k)$
in the limit of large $n$. We begin by mentioning a line of research
in which the question of estimating $l^{p}$ norms of the sequence
$\left(\widehat{b_{\lambda}^{n}}(k)\right)_{k\geq0}$ plays a central
role, see Subsection \ref{subsec:l_p-norms-of-} below. Another motivation,
described in Subsection \ref{subsec:Strongly-annular-fonctions} is
the construction of so called strongly annular functions with specific
decay of the Taylor coefficients.

\subsubsection{\label{subsec:l_p-norms-of-} $l^{p}$ norms of $\widehat{b_{\lambda}^{n}}$
for $p\in[1,\infty]$ and related topics}

$\:$

We use standard notation from asymptotic analysis: From now on,
for two positive functions $f,g\::\:\mathbb{C}\rightarrow\mathbb{R}^{+}$
we say that $f$ is dominated by $g$, denoted by $f\lesssim g$,
if there is a constant $c>0$ such that $f\leq cg$. We say that $f$
and $g$ are comparable, denoted by $f\asymp g$, if both $f\lesssim g$
and $g\lesssim f$. 
\begin{enumerate}
\item The study of the $l^{p}$ norms of $\widehat{b_{\lambda}^{n}}$ was
probably initiated by J.-P.~Kahane~\cite{Kah} who was interested
in the case $p=1$. He applied van der Corput type estimates on $\widehat{b_{\lambda}^{n}}(k)$
\cite[p. 253]{Kah} to get information on the asymptotic behavior
of the $l^{1}$ norm of $\widehat{b_{\lambda}^{n}}$ 
\[
\norm{\widehat{b_{\lambda}^{n}}}{1}:=\sum_{k\geq0}\abs{\widehat{b_{\lambda}^{n}}(k)}.
\]
Kahane's motivation \cite[Theorem 1]{Kah} was to generalize a theorem
by~Z. K. Leibenson \cite{Leib}, which is a special case of a theorem~\cite[Theorem 4.1.3]{Rud}
about homomorphisms of group algebras due to P.~T.~Cohen. Let $\varphi:\:\mathbb{R}\rightarrow\mathbb{R}$
be a continuous, non-constant and $2\pi$-periodic function. A.~Beurling
and H.~Helson~\cite{BeHe} proved that if $\norm{\widehat{e^{in\varphi}}}{1}=\cO\left(1\right),$
$n\in\mathbb{Z}$, then $\varphi$ is affine. Kahane proved that: 
\begin{enumerate}
\item If $\varphi$ is piecewise linear, then $\norm{\widehat{e^{in\varphi}}}{1}=\cO\left(\log(\abs{n})\right)$,~\cite[Theorem III]{Kah}
and\emph{ } 
\item if $\varphi$ is analytic, then $\norm{\widehat{e^{in\varphi}}}{1}\asymp\sqrt{\abs{n}}$,
\cite[Theorem V]{Kah}. 
\end{enumerate}
Writing $b_{\lambda}(e^{it})$ as $e^{i\varphi(t)}$ for $t\in\left(-\pi,\,\pi\right]$, we deduce from (b) that 
\[
\norm{\widehat{b_{\lambda}^{n}}}{1}\sim c_{1}\sqrt{n},\qquad n\rightarrow\infty.
\]
The precise value $c_{1}$ of the limit 
\[
\lim_{n\rightarrow\infty}n^{-1/2}\norm{\widehat{b_{\lambda}^{n}}}{1}
\]
was computed in~\cite{Gir}. A discussion on $l^{p}$ norms for $p\in[1,\infty]$
occurred in~\cite{BlSh}, where the asymptotic behavior 
\begin{equation}
\norm{\widehat{b_{\lambda}^{n}}}{p}\asymp n^{\frac{2-p}{2p}}\ \textnormal{for}\ p\in[1,2]\label{eq:BS}
\end{equation}
is derived. The discussion in~\cite{BlSh} is more general and motivated
by investigating the boundedness of the composition operator $C_{b}$, $C_{b}(f)=f\circ b$, 
where $b=b_{\lambda}$. To assess whether $C_{b}$ is a bounded linear
operator from one Banach space $X$ of analytic functions into another,
say $Y$, it is often enough to know the asymptotic behavior of $\Norm{b_{\lambda}^{n}}{Y}$.\\
 Let us mention that the asymptotic formulas for $\widehat{b_{\lambda}^{n}}(k)$
obtained in the present paper could be used to compute the exact values
of $c_{p}$ defined as follows: 
\begin{enumerate}
\item If $p\in(0,4)$, then 
\[
\lim_{n\rightarrow\infty}n^{-\frac{2-p}{2p}}\norm{\widehat{b_{\lambda}^{n}}}{p}=c_{p},
\]
\item if $p=4$, then 
\[
\lim_{n\rightarrow\infty}\left(\frac{n}{\log n}\right)^{1/4}\norm{\widehat{b_{\lambda}^{n}}}{4}=c_{4},
\]
\item and if $p\in(4,\infty]$, then 
\[
\lim_{n\rightarrow\infty}n^{-\frac{1-p}{3p}}\norm{\widehat{b_{\lambda}^{n}}}{p}=c_{p},
\]
\end{enumerate}
which generalizes Girard's result \cite{Gir} and strengthens \cite[Theorem 1]{SzZa1}.
The constants $c_{p}$ are not studied in this article; their computations
are part of a forthcoming work. 

\item O. Szehr and R. Zarouf proved upper and lower bounds on $\abs{\widehat{b_{\lambda}^{n}}(k)}$
\cite{SzZa1} to complete the result of M. Blyudze and S. Shimorin \eqref{eq:BS}
on $l^{p}$ norms of the sequence $\widehat{b_{\lambda}^{n}},$ extending
\eqref{eq:BS} to the range $p\in[1,4)$ and providing sharp estimates
on $\norm{\widehat{b_{\lambda}^{n}}}{p}$ for the remaining range
$p\in[4,\infty]$. Later on, Szehr and Zarouf \cite[Proposition 2]{SzZa1}
applied those results to estimate analytic capacities in Beurling--Sobolev
spaces. Finally, the same authors \cite{SzZa2,SzZa4} proved upper
bounds on $\abs{\widehat{(1-z^{2})b_{\lambda}^{n}}(k)}$ to construct
a class of counterexamples to Schäffer's conjecture on optimal estimates
for norms of inverses of matrices \cite{Sch,GMP,Que1,Que2}. \\
 Namely, in 1970 J.J.~Schäffer \cite[Theorem 3.8]{Sch} proved that
for any invertible $n\times n$ matrix $T$ and for any operator norm
$\Norm{\cdot}{}$ the inequality 
\[
\abs{\det{T}}\Norm{T^{-1}}{}\leq\mathcal{S}\Norm{T}{}^{n-1}
\]
holds with $\cS=\cS(n)\leq\sqrt{en}$. He conjectured that in fact
this inequality holds with an $\cS$ independent of $n$. This conjecture
was refuted in the early 1990-s by E.~Gluskin, M.~Meyer and A.~Pajor
\cite{GMP} who have shown that for certain $T=T(n)$ the inequality
can only hold when $\cS$ is growing with $n$. Subsequent contributions
of J.~Bourgain \cite{GMP} and H.~Queffélec \cite{Que1,Que2} provided
increasing lower estimates on $\cS$. The currently best known lower
estimate on $\cS$ is due to H.~Queffélec~\cite{Que2} : 
\[
\cS\geq\sqrt{n}(1-\cO(1/n)).
\]
Those results rely on probabilistic and number theoretic arguments.
The common point in the mentioned lower bounds is that they rely on
an inequality of Bourgain \cite[Inequality (2.2)]{SzZa4} that relates
Schäffer's problem to a geometric property of the spectrum of $T$:
For $\cS$ to grow the eigenvalues of $T$ should satisfy a Turán-type
power sum inequality. The construction of explicit solutions to such
inequalities appears to be a well-studied but open problem in number
theory~\cite{Tur,Mon,ErRe,And1,And2}. More precisely, Bourgain's
inequality relates Schäffer's question to Turán's tenth problem~\cite{And1,Tur}.
The latter has no constructive solution and relies on deep number-theoretic
existence arguments~\cite{And1,Mon,Que2}. In \cite{GMP} as well
as in \cite[question 5]{Que1} the construction of explicit matrices
with growing $\cS$ is formulated as an open problem. Constructive
counterexamples to Schäffer's conjecture are proposed in \cite{SzZa4}
where the authors present an explicit sequence of \textit{Toeplitz
matrices} $T_{\lambda}\in\cM_{n}$ with \emph{singleton spectrum}
$\{\lambda\}\in\mathbb{D}\setminus\{0\}$ such that 
\[
\cS\geq\abs{\lambda}^{n}\Norm{T_{\lambda}^{-1}}{}\geq c(\lambda)\sqrt{n}\Norm{T_{\lambda}}{}{}^{n-1},
\]
$c(\lambda)>0.$ The authors use a duality method to prove an analog
of Bourgain's inequality and thereby estimate $\Norm{T_{\lambda}^{-1}}{}$
from below. Their lower bound on $\Norm{T_{\lambda}^{-1}}{}$ involves
the $l^{\infty}$ norm of the sequence 
\[
\widehat{(1-z^{2})b_{\lambda}^{n}}(k)=\widehat{b_{\lambda}^{n}}(k)-\widehat{b_{\lambda}^{n}}(k-2),\qquad k\geq2.
\]
Better numerical estimates on $\abs{\det{T_{\lambda}}}\Norm{T_{\lambda}^{-1}}{}$
can be obtained by considering more elaborate test functions than
the simplest one they chose \cite[Remark 10 and Remark 15]{SzZa4}.
Exact asymptotic expansions for $\widehat{b_{\lambda}^{n}}(k)$ are
therefore of interest to derive numerical lower estimates on $\mathcal{S}/\sqrt{n}$
as $n\rightarrow\infty.$ \\
\end{enumerate}

\subsubsection{\label{subsec:Strongly-annular-fonctions}Strongly annular fonctions}

A function $f$ analytic in the unit disc is said to be annular if
there exists an embedded sequence of open domains $\Omega_{n}$, $\Omega_{n}\subset\Omega_{n+1}$,
$\overline{\Omega_{n}}\subset\mathbb{D}$, $n\ge1$ such that $\cup_{n\ge1}\Omega_{n}=\mathbb{D}$
and 
\[
\lim_{n\to\infty}\min_{\partial\Omega_{n}}|f|=\infty.
\]
Such a function $f$ does not belong to the Nevanlinna class, and,
in particular, it does not belong to the Hardy space $H^{2}$, that
is $\widehat{f}\not\in\ell^{2}$. The function $f$ is said to be
strongly annular if it is annular with $\Omega_{n}=\mathcal{D}(0,r_{n})$,
$r_{n}\to1$. The short book of Bonar \cite{Bon} dedicated to this
subject contains several constructions of such functions coming back,
in particular, to Lusin--Privalov, 1925, Paley, 1930, and Littlewood,
1944.

Let us also mention here some more recent results on strongly annular
functions. In 1997 Daquila \cite{Daq1} studied strongly annular solutions
of Mahler's functional equation and in 2010 he studied \cite{Daq2}
the density of such solutions in the space $\mathcal{H}(\mathbb{D})$
of the functions holomorphic in the unit disc. In 2007 Redett \cite{Red}
constructed strongly annular functions in standard Bergman spaces.
In 2013 Bernal--González--Bonilla \cite{BeBo} proved that the
set of the strongly annular functions is algebraically large (maximal
dense--lineable and algebrable in $\mathcal{H}(\mathbb{D})$). For
random strongly annular functions see, for example, \cite[Chapter 13, Theorem 7]{Kah1}
and \cite{How}.

\subsection{Known results}

\subsubsection{Estimates on $\widehat{b_{\lambda}^{n}}(k)$}

$\:$

Below we recall the known upper/lower bounds on $\widehat{b_{\lambda}^{n}}(k)$
as well as the known asymptotic formulas for these coefficients. 
\begin{enumerate}
\item D. D. Bonar, F. Carroll, and G. Piranian \cite[Theorem 1]{BCP} proved
that there exist positive numbers $A_{1}$ and $A_{2}$ such that
for all $k$ and $n$ the coefficients $\widehat{b_{1/2}^{n}}(k)$
satisfy the inequality 
\[
\abs{\widehat{b_{1/2}^{n}}(k)}\leq A_{1}n^{-1/3}
\]
and such that, for every nonnegative integer $j$, 
\[
\liminf_{n\rightarrow\infty}n^{1/3}\abs{\widehat{b_{1/2}^{n}}(3k+j)}>A_{2}.
\]
\item It is also shown in \cite[Theorem 2]{BCP} that if $k<n/3,$ then
\[
\abs{\widehat{b_{1/2}^{n}}(k)}\leq\frac{6}{\pi}\frac{1}{n-3k},
\]
and that if $k>3n,$ then 
\[
\abs{\widehat{b_{1/2}^{n}}(k)}\leq\frac{2}{\pi}\frac{1}{k-3n}.
\]
\item Szehr--Zarouf \cite[Proposition 2]{SzZa1} proved that if $\alpha\in(0,\alpha_{0})$,
$\alpha_{0}:=\frac{1-\lambda}{1+\lambda}$, then the following assertions
hold for large enough $n$. 
\begin{enumerate}
\item If $k/n\leq\alpha,$ then $\abs{\widehat{b_{\lambda}^{n}}(k)}$ decays
exponentially as $n$ tends to $\infty$, i.e. there exists $q\in(0,1)$
depending on $\alpha$ and $\lambda$ only such that 
\[
\abs{\widehat{b_{\lambda}^{n}}(k)}\leq q^{n}.
\]
Similarly, if $k/n\geq\alpha^{-1}$ then $\abs{\widehat{b_{\lambda}^{n}}(k)}$
decays exponentially as $n$ tends to $\infty$. 
\item If $k/n\in(\alpha,\alpha_{0}-n^{-2/3})\cup(\alpha_{0}^{-1}+n^{-2/3},\alpha^{-1})$
then 
\[
\abs{\widehat{b_{\lambda}^{n}}(k)}\lesssim\max\left\{ \frac{1}{\abs{\alpha_{0}n-k}},\frac{1}{\abs{\alpha_{0}^{-1}n-k}}\right\} .
\]
\item If $k/n\in[\alpha_{0}-n^{-2/3},\alpha_{0}+n^{-2/3})\cup(\alpha_{0}^{-1}-n^{-2/3},\alpha_{0}^{-1}+n^{-2/3}]$
then 
\[
\abs{\widehat{b_{\lambda}^{n}}(k)}\lesssim\frac{1}{n^{1/3}}.
\]
\item If $k/n\in(\alpha_{0}+n^{-2/3},\alpha_{0}^{-1}-n^{-2/3})$ then 
\[
\abs{\widehat{b_{\lambda}^{n}}(k)}\lesssim\max\left\{ \frac{1}{n^{1/2}\abs{\alpha_{0}-\frac{k}{n}}^{1/4}},\frac{1}{n^{1/2}\abs{\alpha_{0}^{-1}-\frac{k}{n}}^{1/4}}\right\} .
\]
\end{enumerate}
\item An asymptotic expansion of $\widehat{b_{\lambda}^{n}}(k)$ as $k$
and $n$ tend simultaneously to $\infty$ and $k$ approaches the
right boundary of $[\alpha_{0}n,\,\alpha_{0}^{-1}n]$ from inside,
i.e. $\lim_{n\rightarrow\infty}\left(\alpha_{0}^{-1}-k/n\right)=0^{+},$
is computed in \cite[Proposition 6]{SzZa1}. In this region the asymptotic
behavior of $\widehat{b_{\lambda}^{n}}(k)$ can be written in terms
of the Airy function $Ai(x)$. For real arguments the latter can be
defined as an improper Riemann integral 
\[
Ai(x)=\frac{1}{\pi}\int_{0}^{\infty}\cos\left(\frac{t^{3}}{3}+xt\right)\d t.
\]
The authors in \cite{SzZa1} were interested in the oscillatory behavior
of $Ai$ for large negative arguments for which we have the asymptotic
approximation: 
\begin{equation}
Ai(-x)\sim\frac{1}{x^{1/4}\sqrt{\pi}}\cos\left(\frac{2}{3}x^{3/2}-\frac{\pi}{4}\right),\ x\rightarrow+\infty.\label{eq:osc_beh_Airy}
\end{equation}
More precisely it is shown in \cite[Proposition 6]{SzZa1} (making
use of a uniform version of the method of stationary phase, see, for example, \cite[Section 2.3]{Bor})
that for sequences $k=k(n)$ with $k\in[\alpha_{0}n,\alpha_{0}^{-1}n]$
such that $\lim_{n\rightarrow\infty}\frac{k}{n}=\alpha_{0}^{-1}$,
the following asymptotic formula holds as $n\rightarrow\infty$ 
\[
\widehat{b_{\lambda}^{n}}(k)\sim\frac{(1-\lambda)^{1/4}}{\left(\lambda(1+\lambda)\right)^{1/12}}\frac{\sqrt{2}}{\sqrt{k/n}\left(k/n-\alpha_{0}\right)^{1/4}}\frac{Ai(n^{2/3}\gamma^{2})}{n^{1/3}},
\]
where 
\[
\gamma^{2}=\gamma_{\alpha_{0}^{-1}}^{2}\sim\frac{1-\lambda}{\left(\lambda(1+\lambda)\right)^{1/3}}\left(k/n-\alpha_{0}^{-1}\right).
\]
We will see that the above asymptotic formula for $\widehat{b_{\lambda}^{n}}(k)$
remains valid also when $k/n$ approaches $\alpha_{0}^{-1}$ from
outside of the compact interval $[\alpha_{0},\alpha_{0}^{-1}]$, see
below~Theorem~\ref{Th_Regions_I_VII}~(4). When $k/n>\alpha_{0}^{-1}$
and $\lim_{n\rightarrow\infty}\frac{k}{n}=\alpha_{0}^{-1}$ we will
use the fact that the Airy function has exponential asymptotics for
large positive arguments 
\begin{equation}
Ai(x)\sim\frac{1}{2x^{1/4}\sqrt{\pi}}\exp\left(-\frac{2}{3}x^{3/2}\right),\ x\rightarrow+\infty.\label{eq:exp_beh_Airy}
\end{equation}
Let us finally mention that in what follows, Theorem~\ref{Th_Regions_I_VII}~(4) and Theorem \ref{Th:Regions_IV_V_VI}~(1),~(2), show a similar asymptotic formula for $\widehat{b_{\lambda}^{n}}(k)$
as $k/n$ approaches the left boundary $\alpha_{0}$ (both from the
left and the right). 
\end{enumerate}

\subsubsection{Strongly annular fonctions }

Most of the known examples of strongly annular functions involve lacunary
series. Frequently, the Taylor coefficients of the functions in such
examples are unbounded. That is why Bonar asked in \cite[Question 6.9]{Bon}
whether every strongly annular function is a sum of a bounded function
and the sum of a lacunary Taylor series. In 1977 Bonar, Carroll, and
Piranian \cite{BCP} constructed a strongly annular function $f$
such that $\lim_{n\to\infty}\widehat{f}(n)=0$ and 
\[
\sum_{n\ge0}\min(|\widehat{f}(2n)|^{2},|\widehat{f}(2n+1)|^{2})=\infty.
\]
In other words, if $s_{k}$ are positive integers, $s_{k+1}>s_{k}+1$,
then 
\[
\sum_{n\ge0,\,n\not\in(s_{k})}\widehat{f}(n)z^{n}\not\in H^{2}.
\]
This construction was based on the above mentioned estimates of the
asymptotics of the Taylor coefficients of $b_{\lambda}^{n}$.

Another construction of strongly annular functions whose Taylor coefficients
tend to $0$ was given by Bonar, Carroll, and Erdös in \cite{BCE}.

\subsection{\label{subsec:Goals-of-the}Goals of the paper}

\subsubsection{\label{subsec:Asymptotic-analysis-of}Asymptotic analysis of $\widehat{b_{\lambda}^{n}}(k)$
as $n\rightarrow\infty$}

$\:$

The first goal of this paper is to state all asymptotic formulas for
$\widehat{b_{\lambda}^{n}}(k)$ as $n\rightarrow\infty$ depending
on the region to which $k=k(n)\in[0,\infty)$ belongs. The above mentioned
upper bounds on $\widehat{b_{\lambda}^{n}}(k)$ are usually based
on van der Corput type estimates, and the standard \textit{Laplace-type
methods} which we describe below, will be used to derive exact asymptotic
formulas for $\widehat{b_{\lambda}^{n}}(k)$ as $n\rightarrow\infty$.
We write the integral defining $\widehat{b_{\lambda}^{n}}(k)$ in
a way that is convenient for asymptotic analysis: 
\begin{equation}
\widehat{b_{\lambda}^{n}}(k)=\frac{1}{2i\pi}\int_{\partial\mathbb{D}}e^{n\Phi(z)}\frac{\d z}{z}\label{eq:integral}
\end{equation}
(the so-called complex Laplace-type integral) where 
\begin{equation}
\Phi(z)=\Phi_{k/n}(z)=\log\left(z^{-\frac{k}{n}}b_{\lambda}(z)\right),\label{Phi}
\end{equation}
and $\log$ denotes a branch of the complex logarithm chosen in the following way: 
if $k/n\le c<\alpha^{-1}_{0}$, then we can take the branch cut $[0,\infty)$ and fix $\log(-1)=i\pi$, and if 
$k/n\ge c>\alpha_{0}$, then we can take the principal branch of the complex logarithm. 
In particular, if $\alpha_{0}<c_1\le k/n\le c_2<\alpha^{-1}_{0}$, then we could take either of these two 
definitions.
 The asymptotic behavior of this integral is studied using standard
tools of asymptotic analysis: the \textit{method of stationary phase}
\cite{Erd,Fed1,Fed2,Bor} or the\textit{ method of the steepest descent}
\cite{BlHa,Bru,Cop,Tem}, depending on the location of the critical
points of $\Phi$. To apply the method of stationary phase we need
to introduce the real function 
\begin{equation}
h(\varphi)=h_{k/n}(\varphi):=-i\Phi_{k/n}(e^{i\varphi})=\arg\left(\left(z^{-k/n}b_{\lambda}(z)\right)_{\vert z=e^{i\varphi}}\right)\qquad\varphi\in[0,\pi],\label{eq:h}
\end{equation}
observing that $\Abs{z^{-k/n}b_{\lambda}(z)}=1$ for $z\in\partial\mathbb{D}$,
so that 
\begin{align}
\widehat{b_{\lambda}^{n}}(k)=\frac{1}{\pi}\Re{\left\{ \int_{0}^{\pi}e^{inh_{k/n}(\varphi)}\d\varphi\right\} .}\label{eq:int_stat_phase}
\end{align}
As usual, the dominant contribution to integrals of the form~\eqref{eq:integral}
(respectively \eqref{eq:int_stat_phase}) comes from a small neighborhood
around the stationary points of $\Phi$ (respectively $h$). We refer
to Lemma \ref{lem:critical_pts_f} below for an identification of
the critical points of $\Phi$ which we denote by $z_{\pm}$, see
also \cite[Section 6]{SzZa2}. It turns out that when $a=k/n\in[\alpha_{0},\alpha_{0}^{-1}]$
we have $z_{\pm}\in\partial\mathbb{D}$ and the integral \eqref{eq:integral}
is especially suited for an application of the method of stationary
phase \cite{Erd,Fed1,Fed2}, whereas if $a=k/n\notin[\alpha_{0},\alpha_{0}^{-1}]$,
then $z_{\pm}\in\mathbb{R}$ and this method fails. In this case, a
deformation of the contour $\partial\mathbb{D}$ will be required
in order to apply the method of the steepest descent. As $k/n$ approaches
one of the boundaries $\alpha_{0}$ or $\alpha_{0}^{-1}$, uniform
versions of these methods \cite[Section 2.3]{Bor} \cite[Section 9.2]{BlHa}
\cite[p. 366--372]{Won} (all of them being based on \cite{CFU})
will be required, see Proposition \ref{Prop_Airy_fcts} below. A summary
of the asymptotics of $\widehat{b_{\lambda}^{n}}(k)$ is provided in Figure~\ref{Table}
below, depending on $k$. The asymptotic formulas for $\widehat{b_{\lambda}^{n}}(k)$
are discussed in full detail in Section \ref{sec:Asymptotic-formulas-for},
see Theorem \ref{Th_Regions_I_VII} and Theorem \ref{Th:Regions_IV_V_VI}.

\subsubsection{\textcolor{blue}{\label{subsec:SAF_content_of_the_paper}} Strongly
annular functions}

Using the ideas from \cite{BCP} and \cite{BCE} and estimates on
the asymptotics of $\widehat{b_{\lambda}^{n}}(k)$ obtained in our
paper, we construct strongly annular functions $f$ such that (a)
$\widehat{f}$ belongs to $\ell^{q}\setminus\ell^{p}$ for any given
$2\le p<q$ or (b) $\widehat{f}$ belongs to $\ell_{\varphi}^{2}\setminus\ell^{2}$
where $\ell_{\varphi}^{2}$ is the set of sequences $(a_{n})$ such
that 
\[
\sum_{n\ge0}|a_{n}|^{2}/\varphi(1/|a_{n}|)<\infty,
\]
and $\varphi$ is such that $\lim_{t\rightarrow\infty}\varphi(t)=\infty.$
Furthermore, the functions $f$ we construct are not lacunary in the
sense that if $(s_{k})$ is a sequence of positive integers such that
$s_{k+1}>s_{k}+1$, then $\widehat{f}\cdot\chi_{\mathbb{Z}_{+}\setminus(s_{k})}\not\in\ell^{p}$
and $\widehat{f}\cdot\chi_{\mathbb{Z}_{+}\setminus(s_{k})}\not\in\ell^{2}$,
correspondingly, in the cases (a) and (b).

\subsection{Outline of the paper}

In Section \ref{sec:Asymptotic-formulas-for} below, we state asymptotic
formulas for $\widehat{b_{\lambda}^{n}}(k)$ as $n\rightarrow\infty$.
We distinguish seven regions of $k$ where the asymptotic behavior
of $\widehat{b_{\lambda}^{n}}(k)$ differs. Given $\alpha\in[\epsilon,\alpha_{0})$
we compute an asymptotic formula for $\widehat{b_{\lambda}^{n}}(k)$
when $k\in[0,\alpha n]\cup[n/\alpha,\infty]$ and thereby sharpen
the known fact asserting that $\widehat{b_{\lambda}^{n}}(k)$ decays
exponentially for $k$ in those regions, see Theorem~\ref{Th_Regions_I_VII}
below. Given $\beta\in(\alpha_{0},\alpha_{0}^{-1})$ we find that
for $k\in[\beta n,n/\beta]$ the asymptotic of $\widehat{b_{\lambda}^{n}}(k)$
is oscillatory and witnesses a decay of order $\cO(n^{-1/2})$, see
Theorem \ref{Th:Regions_IV_V_VI}~(2) below. We also compute
an asymptotic formula for $\widehat{b_{\lambda}^{n}}(k)$ as $k$
and $n$ tend simultaneously to $\infty$ and $k$ approaches the
boundaries $\alpha_{0}n,\,\alpha_{0}^{-1}n$. In these regions the
asymptotic behavior of $\widehat{b_{\lambda}^{n}}(k)$ is described
in terms of the Airy function $Ai(x)$, see Theorem \ref{Th_Regions_I_VII}~(3), (4), Theorem \ref{Th:Regions_IV_V_VI} and Proposition
\ref{Prop_Airy_fcts} for more details. We end Section \ref{sec:Asymptotic-formulas-for}
summing up $\widehat{b_{\lambda}^{n}}(k)$'s asymptotics depending
on the region where $k$ belongs, see Figure~\ref{Table} below.
The proofs of Theorem \ref{Th_Regions_I_VII}, Theorem \ref{Th:Regions_IV_V_VI}
and Proposition \ref{Prop_Airy_fcts} are collected in Section \ref{sec:Proofs_asymptotics}.
In Section~\ref{Annular} we give two constructions of strongly annular
functions with small Taylor coefficients in Theorems~\ref{thmlp}
and \ref{thmlphi}. These constructions are based on auxiliary Lemmas~\ref{lem1}
and \ref{lem2} concerning, correspondingly, properties of $b_{1/2}^{n}$
and flat polynomials.

\section{\label{sec:Asymptotic-formulas-for}Asymptotic formulas for $\widehat{b_{\lambda}^{n}}(k)$}

It is known \cite[Proposition 2]{SzZa1}, \cite[Lemma 7]{SzZa3} that
given $\alpha<\alpha_{0}$, $\widehat{b_{\lambda}^{n}}(k)$ decays
exponentially for $k\in[0,\alpha n]\cup[\alpha^{-1}n,+\infty)$ as
$n$ tends to $+\infty$. Theorem \ref{Th_Regions_I_VII} below sharpens
the previous results in \cite[Theorem 2]{BCP}, \cite[Proposition 2]{SzZa1},
\cite[Lemma 7]{SzZa3} by stating asymptotic formulas for $\widehat{b_{\lambda}^{n}}(k)$
as $n$ tends to $+\infty$ when $k$ belongs to those regions: 
\begin{enumerate}
\item If $k$ is fixed (Region I), then the proof of the asymptotic formula
for $\widehat{b_{\lambda}^{n}}(k)$ follows by induction on $k$. 
\item If $k=k(n)\rightarrow\infty$ as $n\rightarrow\infty$ with $k\leq\alpha n$
(Region II) or $k\geq\alpha^{-1}n$ (Region VIII), then the integral
defining $\widehat{b_{\lambda}^{n}}(k)$ is treated by a direct application
of the method of the steepest descent \cite[Chapter 7]{BlHa}, \cite[Chapters 7-8]{Cop},
\cite[Chapter 5]{Bru}, \cite[Chapter 4]{Tem}, which we will use
intensively in our proof. 
\item If $k\in[\alpha n,\alpha_{0}n-n^{1/3})$ and in addition $n^{2/3}(\alpha_{0}-k/n)\rightarrow+\infty$
(Region III) or if $k\in[\alpha_{0}^{-1}n+n^{1/3},\alpha^{-1}n]$
and in addition $n^{2/3}(k/n-\alpha_{0}^{-1})\rightarrow+\infty$
(Region VII), then a uniform version of the steepest descent method
based on \cite{CFU}, see \cite[Section 9.2]{BlHa}, \cite[p. 366--372]{Won},
is required to obtain the asymptotic formula for $\widehat{b_{\lambda}^{n}}(k)$.
More precisely, the proof of our asymptotic formulas for $k$ in Regions
III and VII will follow from an application of Proposition \ref{Prop_Airy_fcts}
(stated below) together with the approximation \eqref{eq:exp_beh_Airy}
of the Airy function for large positive arguments. 
\end{enumerate}
Our asymptotic formulas witnessing exponential decay of $\widehat{b_{\lambda}^{n}}(k)$
for $k$ in Regions I-II-III-VII-VIII are sharp, new and agree on
the intersections of Regions I-II, II-III and VII-VIII: We refer to
the comments below just afer the statement of Theorem \ref{Th_Regions_I_VII}
for a detailed discussion, where we also compare our results to the
previous upper estimates from \cite[Proposition 2]{SzZa1}. We recall
that the value of $\alpha_{0}$ is given by $\alpha_{0}=\frac{1-\lambda}{1+\lambda}$
and that $\Phi$ is defined according to \eqref{Phi} by 
\[
\Phi(z)=\Phi_{k/n}(z)=\log\left(z^{-\frac{k}{n}}b_{\lambda}(z)\right).
\]

\begin{thm}
\label{Th_Regions_I_VII} \label{Exponentiel decay of the Fourier coefficients of the powers of a Blaschke factor}
Let $\alpha\in(0,\alpha_{0})$. Consider a sequence $\omega(n^{1/3})$
such that $\omega(n^{1/3})/n^{1/3}\rightarrow\infty$ as $n\rightarrow\infty$
and assume additionally that $\omega(n^{1/3})=o(n)$ as $n\rightarrow\infty$.
The following asymptotic formulas for the $k^{{\rm th}}-$Fourier
coefficients of $b_{\lambda}^{n}$ hold as $n$ tends to $+\infty$.

{\rm (1)} If $k$ is fixed (Region I), then 
\[
\widehat{b_{\lambda}^{n}}(k)\sim\frac{(-\lambda)^{n-k}\left(n(1-\lambda^{2})\right)^{k}}{k!}.
\]

{\rm (2)} If $k=k(n)\rightarrow\infty$ as $n\rightarrow\infty$ with $k\leq\alpha n$
(Region II) or $k\geq\alpha^{-1}n$ (Region VIII), then 
\[
\widehat{b_{\lambda}^{n}}(k)\sim\frac{1}{\sqrt{2k\pi}}\frac{1}{\left[(\alpha_{0}-k/n)(\alpha_{0}^{-1}-k/n)\right]^{1/4}}\left(\frac{b_{\lambda}(z_{+})}{z_{+}^{k/n}}\right)^{n},
\]
where $z_{+}$ is defined by 
\begin{equation}
z_{+}=z_{+}(k/n)=\frac{\frac{k}{n}(1+\lambda^{2})-(1-\lambda^{2})}{2\lambda\frac{k}{n}}+\sqrt{\left(\frac{\frac{k}{n}(1+\lambda^{2})-(1-\lambda^{2})}{2\lambda\frac{k}{n}}\right)^{2}-1}.\label{eq:z_p}
\end{equation}

{\rm (3)} If $k\in[\alpha n,\alpha_{0}n-\omega(n^{1/3})]$ (Region III),
then 
\[
\widehat{b_{\lambda}^{n}}(k)\sim\frac{(-1)^{n-k}}{\sqrt{2n\pi}}\frac{1}{\sqrt{k/n}\left[(\alpha_{0}^{-1}-k/n)(\alpha_{0}-k/n)\right]^{1/4}}\exp\left(-\frac{2}{3}n\abs{\gamma_{{\alpha_{0}}}}^{3}\right),
\]
where $\gamma_{\alpha_{0}}^{3}$ is given by 
\begin{equation}
\gamma_{\alpha_{0}}^{3}=\frac{3}{2}\left[\Phi(z_{+})-i\pi\left(1-\frac{k}{n}\right)\right],\label{eq:exact_g_3_left_bd}
\end{equation}
and in particular 
\begin{equation}
\gamma_{\alpha_{0}}^{3}\sim-\frac{\left(\alpha_{0}-k/n\right)^{3/2}(1+\lambda)^{3/2}}{(\lambda(1-\lambda))^{1/2}},\qquad k/n\to\alpha_{0},\qquad k/n<\alpha_{0}.\label{eq:g_3_left_approx}
\end{equation}

{\rm (4)} If $k\in[\alpha_{0}^{-1}n+\omega(n^{1/3}),\alpha^{-1}n]$ (Region
VII), then 
\[
\widehat{b_{\lambda}^{n}}(k)\sim\frac{1}{\sqrt{2n\pi}}\frac{1}{\sqrt{k/n}\left[(k/n-\alpha_{0}^{-1})(k/n-\alpha_{0})\right]^{1/4}}\exp\left(-\frac{2}{3}n\abs{\gamma_{{\alpha_{0}}^{-1}}}^{3}\right),
\]
where $\gamma_{\alpha_{0}^{-1}}^{3}$ is given by 
\begin{equation}
\gamma_{\alpha_{0}^{-1}}^{3}=\frac{3}{2}\Phi(z_{+}),\label{eq:exact_g_3_right_bd}
\end{equation}
and in particular 
\begin{equation}
\gamma_{\alpha_{0}^{-1}}^{3}\sim-\frac{(k/n-\alpha_{0}^{-1})^{3/2}(1-\lambda)^{3/2}}{\left(\lambda(1+\lambda)\right)^{1/2}},\qquad k/n\to\alpha_{0}^{-1},\qquad k/n>\alpha_{0}^{-1}.\label{eq:g_3_right_approx}
\end{equation}
\end{thm}

We proceed with a series of remarks and observations highlighting
the coincidence of our formulas for $k$ on the intersections of Regions
I-II, II-III and VII-VIII, and comparing our results to \cite[Proposition 3~(1)]{SzZa1}
and \cite[Proposition 3(2)]{SzZa1}. 
\begin{enumerate}
\item The asymptotic formula stated for $k$ in Region II agrees with the
one for $k$ in Region I. Indeed, a direct computation shows that if
$k$ is fixed, then 
\[
z_{+}^{-k}\sim(-1)^{k}\frac{\left(n(1-\lambda^{2})\right)^{k}}{k^{k}\lambda^{k}},\qquad n\rightarrow\infty,
\]
and 
\[
(b_{\lambda}(z_{+}))^{n}\sim(-1)^{n}e^{k}\lambda^{n},\qquad n\rightarrow\infty.
\]
The coincidence of the asymptotics follows from an application of
Stirling's formula: if $k=o(n),$ then we have 
$$
\widehat{b_{\lambda}^{n}}(k) \sim\frac{1}{\sqrt{2k\pi}}\frac{b_{\lambda}(z_{+})}{z_{+}^{k}}^{n}\sim\frac{(-\lambda)^{n-k}}{\sqrt{2k\pi}}\frac{\left(n(1-\lambda^{2})\right)^{k}}{k^{k}e^{-k}}
  \sim(-\lambda)^{n-k}\frac{\left(n(1-\lambda^{2})\right)^{k}}{k!}.
$$
\item Theorem \ref{Th_Regions_I_VII}~(1),~(2) sharpens the result from
\cite[Proposition 3~(1)]{SzZa1} for $k\in[0,\alpha n]\cup[\alpha^{-1}n,\infty)$.
The latter asserts that $\widehat{b_{\lambda}^{n}}(k)$ decays exponentially
and uniformly for $k\leq\alpha n$ respectively $k\ge\alpha^{-1}n$.
We observe that since the function $a\mapsto\frac{\abs{b_{\lambda}(z_{+}(a))}}{\abs{z_{+}(a)}^{a}}$
is increasing on the interval $[0,\alpha_{0}]$, we have 
\[
\frac{\abs{b_{\lambda}(z_{+}(k/n))}}{\abs{z_{+}(k/n)}^{k/n}}\leq\frac{\abs{b_{\lambda}(z_{+}(\alpha))}}{\abs{z_{+}(\alpha)}^{\alpha}}<1,
\]
and therefore Theorem \ref{Th_Regions_I_VII} gives that 
\[
\abs{\widehat{b_{\lambda}^{n}(k)}}\leq\frac{C}{\sqrt{k}}\left(\frac{\abs{b_{\lambda}(z_{+}(\alpha))}}{\abs{z_{+}(\alpha)}^{\alpha}}\right)^{n}
\]
uniformly for $k\in[0,\alpha n]$, where $C=C(\lambda,\alpha)>0.$
A similar argument for $k\in[\alpha^{-1}n,\infty)$ (Region VIII)
leads to the same conclusion. 
\item The formulation of Theorem \ref{Th_Regions_I_VII}~(3),~(4) includes the number $\gamma^{3}\in\left\{ \gamma_{\alpha_{0}}^{3},\,\gamma_{\alpha_{0}^{-1}}^{3}\right\} $
whose value is given by \cite[formula (9.2.9),  p. 370]{BlHa}:
\begin{equation}
\gamma^{3}=\frac{3}{4}\left[\Phi(z_{+})-\Phi(z_{-})\right]\label{eq:def_g3_gen}
\end{equation}
where $z_{+}$ is defined by \eqref{eq:z_p} and 
\[
z_{-}=z_{-}(k/n)=\frac{\frac{k}{n}(1+\lambda^{2})-(1-\lambda^{2})}{2\lambda\frac{k}{n}}-\sqrt{\left(\frac{\frac{k}{n}(1+\lambda^{2})-(1-\lambda^{2})}{2\lambda\frac{k}{n}}\right)^{2}-1}.
\]
Formulas \eqref{eq:exact_g_3_left_bd}, \eqref{eq:g_3_left_approx},
\eqref{eq:exact_g_3_right_bd}, and \eqref{eq:g_3_right_approx} all
follow from the above definition \eqref{eq:def_g3_gen} of $\gamma^{3}$. 
\item The formulas stated for $k$ belonging to Regions II--III--VII--VIII coincide.
Assuming $k/n<$$\alpha_{0}$ (Region II) -- respectively $k/n>\alpha_{0}^{-1}$
(Region VI) -- we have $\gamma_{\alpha_{0}}^{3}<0$ -- respectively
$\gamma_{\alpha_{0}^{-1}}^{3}<0$ . Therefore 
\begin{align*}
\exp\left(-\frac{2}{3}n\abs{\gamma_{\alpha_{0}}}^{3}\right) & =\exp\left(n\left(\Phi(z_{+})-i\pi\left(1-\frac{k}{n}\right)\right)\right)\\
 & =(-1)^{k-n}\left(\frac{b_{\lambda}(z_{+})}{z_{+}^{k/n}}\right)^{n}
\end{align*}
and 
\begin{align*}
\exp\left(-\frac{2}{3}n\abs{\gamma_{{\alpha_{0}^{-1}}}}^{3}\right) & =\exp\left(n(\Phi(z_{+}))\right)\\
 & =\left(\frac{b_{\lambda}(z_{+})}{z_{+}^{k/n}}\right)^{n},
\end{align*}
which shows that our asymptotic formulas in these four regions are
actually the same. 
\item The asymptotic formulas stated for $k$ in Regions III
and VII, see Theorem \ref{Th_Regions_I_VII}~(3),~(4), significantly
improve the estimate from \cite[Proposition 2,  (2)]{SzZa1} where
the decay of $\widehat{b_{\lambda}^{n}}(k)$ is only shown to be 
\[
\cO\left(\max\left\{ \frac{1}{\abs{\alpha_{0}n-k}},\frac{1}{\abs{\alpha_{0}^{-1}n-k}}\right\} \right).
\]
\end{enumerate}
The following result, Theorem \ref{Th:Regions_IV_V_VI} below, establishes
asymptotic formulas for $\widehat{b_{\lambda}^{n}}(k)$ as $n$ tends
to $+\infty$ when $k$ belongs to the remaining regions where it
turns out that the decay of $\widehat{b_{\lambda}^{n}}(k)$ is no
longer exponential but either of order $\cO(n^{-1/3})$ or oscillatory
and of order $\cO(n^{-1/2})$. More precisely: 
\begin{enumerate}
\item Airy-type behaviour for $\widehat{b_{\lambda}^{n}}(k)$ near the
$k$-transition points $n\alpha_{0}$ and $n\alpha_{0}^{-1}$ is established, 
which extends the formula from \cite[Proposition 6]{SzZa1}
to the case $k>\alpha_{0}^{-1}n$ and generalizes it to the left boundary
$\alpha_{0}n$ (for $k$ both from the left and from the right of
$\alpha_{0}n$). Our asymptotic formulas are respectively given below
for $k$ near $n\alpha_{0}$ (Region IV) see Theorem \ref{Th:Regions_IV_V_VI}~(1), and for $k$ near $n\alpha_{0}^{-1}$ (Region VI), see Theorem
\ref{Th:Regions_IV_V_VI}(3), asserting that the decay of $\widehat{b_{\lambda}^{n}}(k)$
for $k$ in these regions is of order $\cO(n^{-1/3})$ at least when
$k$ lies in neighbourhoods of the boundaries $\alpha_{0}n,\,\alpha_{0}^{-1}n$
of length proportional to $n^{1/3}$. For $k$ in those neighborhoods,
the quantity $n^{2/3}\gamma^{2}$ is always bounded in $n.$ The main
tool to prove Theorem \ref{Th:Regions_IV_V_VI} (1), (3) is the uniform
version of the steepest descent method based on \cite{CFU} already
mentioned above, which we apply following \cite[Section 9.2]{BlHa}
to prove Proposition \ref{Prop_Airy_fcts} as an intermediate step. 
\item If $k$ lies in the remaining central region, $(\alpha_{0}n+n^{1/3},\alpha_{0}^{-1}n-n^{1/3})$,
and if in addition $n^{2/3}(k/n-\alpha_{0})\rightarrow+\infty$ or
$n^{2/3}(\alpha_{0}^{-1}-k/n)\rightarrow+\infty$, we find that the
decay of $\widehat{b_{\lambda}^{n}}(k)$ is oscillatory and of order
$\cO(n^{-1/2})$. The corresponding asymptotic formula is stated below,
see Theorem \ref{Th:Regions_IV_V_VI} (2). To prove the latter for
$k/n\rightarrow\alpha_{0}$ or $k/n\rightarrow\alpha_{0}^{-1}$ we
choose $\beta\in(\alpha_{0},1)$ close enough to $\alpha_{0}$, and
combine a uniform version of the method of stationary phase \cite[Section 2.3]{Bor}
(again based on \cite{CFU}, see the proof of Proposition \ref{Prop_Airy_fcts}
below) with the approximation \eqref{eq:osc_beh_Airy} of the Airy
function for large negative arguments. The proof of Theorem \ref{Th:Regions_IV_V_VI}
(2) for $k$ in the remaining interval $[\beta n,\beta^{-1}n]$ follows
from an application of the standard version of the stationary phase
method \cite[Theorem 4]{Erd}. The proof is however rather long and
technical, and we will use a more elaborate version of this method
due to M.V.~Fedoryuk \cite[Theorem 2.4 p.~80]{Fed2}, \cite[Theorem 1.6 p.107]{Fed1},
which will make the argument much shorter, see Section \ref{subsec:Fedoryuk}
for more details. 
\end{enumerate}
The asymptotic approximations \eqref{eq:exp_beh_Airy} -- respectively
\eqref{eq:osc_beh_Airy} -- for large positive -- respectively negative
-- arguments of the Airy function, show that our asymptotic formulas
coincide for $k$ at the intersection of Regions III-IV, IV-V, V-VI
and VI-VII. 
\begin{thm}
\label{Th:Regions_IV_V_VI} Let $\omega(n^{1/3})$ be a sequence such
that $\omega(n^{1/3})/n^{1/3}\rightarrow\infty$ as $n\rightarrow\infty$.
We assume in addition that $\omega(n^{1/3})=o(n)$ as $n\rightarrow\infty$.
The following asymptotic formulas for the $k^{{\rm th}}$ Fourier
coefficients of $b_{\lambda}^{n}$ hold as $n$ tends to $+\infty$.

{\rm (1)} If $k\in[\alpha_{0}n-\omega(n^{1/3}),\alpha_{0}n+\omega(n^{1/3})]$
(Region IV), then 
\[
\widehat{b_{\lambda}^{n}}(k)\sim\frac{(-1)^{n-k}\sqrt{2}}{n^{1/3}}\frac{(1+\lambda)^{1/4}}{(\lambda(1-\lambda))^{1/12}}\frac{1}{\sqrt{k/n}(\alpha_{0}^{-1}-k/n)^{1/4}}Ai\left(n^{2/3}\gamma_{{\alpha_{0}}}^{2}\right),
\]
where $\gamma_{\alpha_{0}}^{2}$ is asymptotically given by 
\begin{equation}
\gamma_{\alpha_{0}}^{2}\sim\frac{\left(\alpha_{0}-k/n\right)(1+\lambda)}{(\lambda(1-\lambda))^{1/3}},\qquad k/n\to\alpha_{0}.\label{eq:g_2_left}
\end{equation}

{\rm (2)} If $k\in[\alpha_{0}n+\omega(n^{1/3}),\alpha_{0}^{-1}n-\omega(n^{1/3})]$
(Region V), then 
\[
\widehat{b_{\lambda}^{n}}(k)\sim\sqrt{\frac{2}{n\pi}}\frac{\cos\left(nh(\varphi_{+})-\pi/4\right)}{\sqrt{k/n}\left[(\alpha_{0}^{-1}-k/n)(k/n-\alpha_{0})\right]^{1/4}},
\]
where $h=h_{k/n}$ is defined in \eqref{eq:h} and the parameter $\varphi_{+}\in[0,\pi]$
is defined by 
\[
e^{i\varphi_{+}}=z_{+}=\frac{\frac{k}{n}(1+\lambda^{2})-(1-\lambda^{2})}{2\lambda\frac{k}{n}}+i\sqrt{1-\left(\frac{\frac{k}{n}(1+\lambda^{2})-(1-\lambda^{2})}{2\lambda\frac{k}{n}}\right)^{2}}.
\]

{\rm (3)} If $k\in[\alpha_{0}^{-1}n-\omega(n^{1/3}),\alpha_{0}^{-1}+\omega(n^{1/3})]$
(Region VI), then 
\[
\widehat{b_{\lambda}^{n}}(k)\sim\frac{\sqrt{2}}{n^{1/3}}\frac{(1-\lambda)^{1/4}}{(\lambda(1+\lambda))^{1/12}}\frac{1}{\sqrt{k/n}(k/n-\alpha_{0})^{1/4}}Ai\left(n^{2/3}\gamma_{{\alpha_{0}}^{-1}}^{2}\right),
\]
where 
\begin{equation}
\gamma_{\alpha_{0}^{-1}}^{2}\sim\frac{(k/n-\alpha_{0}^{-1})(1-\lambda)}{\left(\lambda(1+\lambda)\right)^{1/3}},\qquad k/n\to\alpha_{0}^{-1}.\label{eq:g_2_right}
\end{equation}
\end{thm}

The formulas given in Theorem \ref{Th:Regions_IV_V_VI} (1) respectively
(3) are actually valid for $k/n$ in a \textit{fixed} neighbourhood
of $\alpha_{0}$ respectively $\alpha_{0}^{-1}$. In fact, they hold
more generally if $k\in[\alpha n,\beta n]$ respectively $k\in[\beta^{-1}n,\alpha^{-1}n]$
as long as $\alpha\in(0,\alpha_{0})$ and $\beta\in(\alpha_{0},1)$
are chosen close enough to $\alpha_{0}$. This is the content of Proposition
\ref{Prop_Airy_fcts} below, which entirely describes the Airy-type
behaviour of $\widehat{b_{\lambda}^{n}}(k)$ near the $k$-transition
points $n\alpha_{0}$ and $n\alpha_{0}^{-1}$. It is a modified version
of \cite[Proposition 17]{SzZa4} where only upper bounds were stated
and where the factor $(1-z^{2})$ has been replaced by 1. The main
tool to prove Proposition \ref{Prop_Airy_fcts} is a result from \cite{CFU},
which we apply following \cite[Section 9.2]{BlHa}. 
\begin{prop}
\label{Prop_Airy_fcts} Fix $\alpha\in(0,\alpha_{0})$ and $\beta\in(\alpha_{0},1)$.
Suppose that $\alpha$ and $\beta$ are close enough to $\alpha_{0}$.
If $k/n\in[\alpha,\beta]$, then 
\[
\widehat{b_{\lambda}^{n}}(k)\sim_{n\rightarrow\infty}(-1)^{n-k}\sqrt{\frac{2\abs{\gamma}}{k/n}}\frac{1}{\abs{\Delta}^{1/4}}\frac{Ai(n^{2/3}\gamma^{2})}{n^{1/3}},
\]
where $\Delta=(k/n-\alpha_{0})(\alpha_{0}^{-1}-k/n)$ and $\gamma^{2}=\gamma_{\alpha_{0}}^{2}$
is asymptotically given by \eqref{eq:g_2_left} as $k/n\rightarrow\alpha_{0}.$
\textup{If $k/n\in[\beta^{-1},\alpha^{-1}]$, then } 
\[
\widehat{b_{\lambda}^{n}}(k)\sim_{n\rightarrow\infty}\sqrt{\frac{2\abs{\gamma}}{k/n}}\frac{1}{\abs{\Delta}^{1/4}}\frac{Ai(n^{2/3}\gamma^{2})}{n^{1/3}},
\]
where $\gamma^{2}=\gamma_{\alpha_{0}^{-1}}^{2}$ is asymptotically
given by \eqref{eq:g_2_right} as $k/n\rightarrow\alpha_{0}^{-1}.$ 
\end{prop}

\begin{rem*}
The factor $(-1)^{n-k}$ in the first formula of Proposition \ref{Prop_Airy_fcts}
corresponds to that in Theorem \ref{Th_Regions_I_VII} for $k$ in
Regions I and II. Indeed, the Airy function is positive in a neighborhood
of 0, and for $k\in[0,\alpha n]$ (Region I) the sign of the factor
$\left(\frac{b_{\lambda}(z_{+})}{z_{+}^{k/n}}\right)^{n}$ is $(-1)^{n-k}$
because $z_{+}=z_{+}(k/n)$ is negative. 
\end{rem*}
Proposition \ref{Prop_Airy_fcts} shows in particular that: 
\begin{enumerate}
\item For $k\in[\alpha n,\alpha_{0}n-n^{1/3})$ respectively $k\in[\alpha_{0}^{-1}n+n^{1/3},\alpha^{-1}n]$
such that $n^{2/3}(\alpha_{0}-k/n)\rightarrow+\infty$ respectively
$n^{2/3}(k/n-\alpha_{0}^{-1})\rightarrow+\infty$, we use \eqref{eq:g_2_left}
respectively \eqref{eq:g_2_right} to observe that $n^{2/3}\gamma^{2}\rightarrow+\infty$
and then use the asymptotic approximation \eqref{eq:exp_beh_Airy}
for large positive arguments of the Airy function, to deduce the precise
nature of $\widehat{b_{\lambda}^{n}}(k)$'s exponential decay in these
regions, see Theorem \ref{Th_Regions_I_VII}~(3),~(4) above. 
\item For $k\in(\alpha_{0}n+n^{1/3},\beta n]\cup[\beta^{-1}n,\alpha_{0}^{-1}n-n^{1/3})$
such that either $n^{2/3}(k/n-\alpha_{0})\rightarrow+\infty$ or $n^{2/3}(\alpha_{0}^{-1}-k/n)\rightarrow+\infty$,
we use \eqref{eq:g_2_left} and \eqref{eq:g_2_right} to observe that
$n^{2/3}\gamma^{2}\rightarrow-\infty$ and then apply the asymptotic
approximation \eqref{eq:osc_beh_Airy} for large negative arguments
of the Airy function, which shows that the decay of $\widehat{b_{\lambda}^{n}}(k)$
is oscillatory in these regions, see Theorem \ref{Th:Regions_IV_V_VI}
(2) above. 
\end{enumerate}

\subsection{Summary of $\widehat{b_{\lambda}^{n}}(k)$'s asymptotics}

The table below, see Figure \ref{Table}, shows values of $A(n,k)$
such that 
\[
\widehat{b_{\lambda}^{n}}(k)\asymp A(n,k)
\]
depending on the region to which $k$ belongs. Again we use Landau
standard notation and denote by $\omega(n^{1/3})$ a sequence such
that $\omega(n^{1/3})/n^{1/3}\rightarrow\infty$ as $n\rightarrow\infty$.
We assume in addition that $\omega(n^{1/3})=o(n)$ as $n\rightarrow\infty$.
The numbers $\gamma_{\alpha_{0}}$ and $\gamma_{{\alpha_{0}}^{-1}}$
are asymptotically given by 
\[
\gamma_{\alpha_{0}}^{2}\asymp\alpha_{0}-k/n\qquad{\rm and}\qquad\gamma_{1/\alpha_{0}}^{2}\asymp k/n-1/\alpha_{0}
\]
respectively as $k/n\to\alpha_{0}$ and $\alpha_{0}^{-1}$. The table
shows that the asymptotic behavior of $\widehat{b_{\lambda}^{n}}(k)$
is symmetric with respect to Region V. A possible explanation for
that symmetry is due to the following observation, which is a consequence
of a simple change of variable. 
\begin{prop}
Given $\lambda\in(0,1),$ $k\geq1$, and $n\geq1$, the following
identity holds 
\[
\widehat{b_{\lambda}^{n}}(k)=\frac{(-1)^{n-k}}{2i\pi}\int_{\partial\mathbb{D}}\tilde{\varphi}(z)\exp\left(k\tilde{\Phi}(z)\right){\rm d}z
\]
where 
\[
\tilde{\varphi}(z)=\frac{1}{z}\frac{1-\lambda^{2}}{\abs{1-\lambda z}^{2}}\qquad{\rm and}\qquad\tilde{\Phi}(z)=\log\left(\frac{b_{\lambda}(z)}{z^{\frac{n}{k}}}\right)=\Phi_{n/k}(z),
\]
$\Phi_{n/k}$ being defined by \eqref{Phi}. 
\end{prop}

\begin{proof}
We first write

\begin{align*}
\widehat{b_{\lambda}^{n}}(k) & =\frac{1}{2i\pi}\int_{\partial\mathbb{D}}\left(b_{\lambda}(z)\right)^{n}z^{-k-1}{\rm d}z\\
 & =\frac{(-1)^{n}}{2i\pi}\int_{\partial\mathbb{D}}\left(\widetilde{b_{\lambda}}(z)\right)^{n}z^{-k-1}{\rm d}z
\end{align*}
where $\widetilde{b_{\lambda}}(z)=-b_{\lambda}(z)=\frac{\lambda-z}{1-\lambda z}$
satisfies $\widetilde{b_{\lambda}}\circ\widetilde{b_{\lambda}}={\rm id}$.
Changing the variable $z$ by $u=\widetilde{b_{\lambda}}(z)$ we get
$z=\widetilde{b_{\lambda}}(u)$, ${\rm d}z=-\frac{1-\lambda^{2}}{(1-\lambda u)^{2}}{\rm d}u$
and therefore 
\begin{align*}
\widehat{b_{\lambda}^{n}}(k) & =-\frac{(-1)^{n}}{2i\pi}\int_{\partial\mathbb{D}}u^{n}\left(\widetilde{b_{\lambda}}(u)\right)^{-k-1}\frac{1-\lambda^{2}}{(1-\lambda u)^{2}}{\rm d}u\\
 & =\frac{(-1)^{n-k}}{2i\pi}\int_{\partial\mathbb{D}}u^{n}\left(b_{\lambda}(u)\right)^{-k-1}\frac{1-\lambda^{2}}{(1-\lambda u)^{2}}{\rm d}u.\\
 & =\frac{(-1)^{n-k}}{2\pi}\int_{\partial\mathbb{D}}u^{n+1}\left(b_{\lambda}(u)\right)^{-k-1}\frac{1-\lambda^{2}}{(1-\lambda u)^{2}}\vert_{u=e^{it}}{\rm d}t
\end{align*}
Taking into account the fact that $\widehat{b_{\lambda}^{n}}(k)$
is real and using complex conjugation we find 
\begin{align*}
\widehat{b_{\lambda}^{n}}(k) & =\frac{(-1)^{n-k}}{2\pi}\int_{\partial\mathbb{D}}\left(b_{\lambda}(u)\right)^{k+1}u^{-n-1}u^{2}\frac{1-\lambda^{2}}{(u-\lambda)^{2}}\vert_{u=e^{it}}{\rm d}t\\
 & =\frac{(-1)^{n-k}}{2\pi}\int_{\partial\mathbb{D}}\left(b_{\lambda}(u)\right)^{k}u^{-n+1}\frac{1-\lambda^{2}}{(u-\lambda)(1-\lambda u)}\vert_{u=e^{it}}{\rm d}t\\
 & =\frac{(-1)^{n-k}}{2i\pi}\int_{\partial\mathbb{D}}\frac{1-\lambda^{2}}{\abs{1-\lambda u}^{2}}\left(b_{\lambda}(u)\right)^{k}u^{-n}\frac{{\rm d}u}{u},
\end{align*}
which completes the proof. 
\end{proof}
\begin{figure}[ht!]
\label{Table} \centering

\begin{tabular}{|c|c|c|}
\hline 
Values of $k(n)$ in interval  & Asymptotics of $\widehat{b_{\lambda}^{n}}(k)$  & Region\tabularnewline
\hline 
\hline 
$[0,\,\alpha n]$  & $\frac{1}{\sqrt{k/n}\left[(\alpha_{0}-k/n)(\alpha_{0}^{-1}-k/n)\right]^{1/4}}\frac{1}{\sqrt{n}}\left(\frac{b_{\lambda}(z_{+})}{z_{+}^{k/n}}\right)^{n}$  & I-II\tabularnewline
\hline 
$(\alpha n,\,\alpha_{0}n-\omega(n^{1/3})]$  & $\frac{1}{\sqrt{k/n}\left[(\alpha_{0}^{-1}-k/n)(\alpha_{0}-k/n)\right]^{1/4}}\text{\ensuremath{\frac{1}{\sqrt{n}}}}e^{-\frac{2}{3}n\abs{\gamma_{{\alpha_{0}}^{-1}}}^{3}}$  & III\tabularnewline
\hline 
$[\alpha_{0}n-\omega(n^{1/3}),\,\alpha_{0}n+\omega(n^{1/3})]$  & $\frac{1}{\sqrt{k/n}(\alpha_{0}^{-1}-k/n)^{1/4}}\frac{Ai(n^{2/3}\gamma_{{\alpha_{0}}}^{2})}{n^{1/3}}$  & IV\tabularnewline
\hline 
$[\alpha_{0}n+\omega(n^{1/3}),\,\alpha_{0}^{-1}n-\omega(n^{1/3})]$  & $\frac{1}{\sqrt{n}}\frac{\cos\left(nh(\varphi_{+})-\pi/4\right)}{\sqrt{k/n}\left[(\alpha_{0}^{-1}-k/n)(k/n-\alpha_{0})\right]^{1/4}}$  & V\tabularnewline
\hline 
$[\alpha_{0}^{-1}n-\omega(n^{1/3}),\,\alpha_{0}^{-1}n+\omega(n^{1/3})]$  & $\frac{1}{\sqrt{k/n}(k/n-\alpha_{0})^{1/4}}\frac{Ai(n^{2/3}\gamma_{{\alpha_{0}}^{-1}}^{2})}{n^{1/3}}$  & VI\tabularnewline
\hline 
$[\alpha_{0}^{-1}n+\omega(n^{1/3}),\,\alpha^{-1}n)$  & $\frac{1}{\sqrt{k/n}\left[(k/n-\alpha_{0}^{-1})(k/n-\alpha_{0})\right]^{1/4}}\frac{1}{\sqrt{n}}e^{-\frac{2}{3}n\abs{\gamma_{{\alpha_{0}}^{-1}}}^{3}}$  & VII\tabularnewline
\hline 
$[\alpha^{-1}n,\,\infty)$  & $\frac{1}{\sqrt{k/n}\left[(\alpha_{0}-k/n)(\alpha_{0}^{-1}-k/n)\right]^{1/4}}\frac{1}{\sqrt{n}}\left(\frac{b_{\lambda}(z_{+})}{z_{+}^{k/n}}\right)^{n}$  & VIII\tabularnewline
\hline 
\end{tabular}\caption{Asymptotic formulas for $\widehat{b_{\lambda}^{n}}(k)$ as $n\rightarrow\infty$,
up to numerical factors. For $k$ in Regions I--II and VIII, we have
$\abs{z_{+}^{-k/n}b_{\lambda}(z_{+})}<1$ and the decay of $\widehat{b_{\lambda}^{n}}(k)$
is exponential. The values $\gamma_{\alpha_{0}}$ and $\gamma_{{\alpha_{0}}^{-1}}$
are asymptotically given by $\gamma_{\alpha_{0}}^{2}\asymp\alpha_{0}-k/n$
and $\gamma_{\alpha_{0}^{-1}}^{2}\asymp k/n-\alpha_{0}^{-1}$ respectively
as $k/n\to\alpha_{0}$ and $\alpha_{0}^{-1}$. The formulas for $k$
in Regions III and VII ensure the transition between the exponential
decay (Regions I--II and VIII) and the $\cO(n^{-1/3})$ decay, which
occurs in Regions IV and VI when the distance between $k$ and $\alpha_{0}n$
respectively $\alpha_{0}^{-1}n$ does not exceed $n^{1/3}$. Finally,
the formula for $k$ in Region V ensures the transition to an oscillatory
decay of order $\cO(n^{-1/2})$ when $k$ is away from the boundaries
$\alpha_{0}n$ and $\alpha_{0}^{-1}n$ (we refer to Theorem \ref{Th:Regions_IV_V_VI}~(2) for the definition of $h(\varphi_{+})$).}
\end{figure}

\section{\label{sec:Proofs_asymptotics}Proofs of the asymptotic formulas
for $\widehat{b_{\lambda}^{n}}(k)$ }

As usual, the dominant contribution to integrals of the form~ \eqref{eq:integral}
comes from a small neighborhood around the stationary points of $\Phi_{a}$.
Therefore we start by recalling the critical
points of $\Phi_{a}$. The following lemma is a more complete version
of \cite[Lemma 11]{SzZa2}. We prove it below for completeness. 
\begin{lem}
\label{lem:critical_pts_f} Let $a=k/n$ and let $\Phi(z)=\Phi_{a}(z)$
be defined as above. We have the following assertions. 
\begin{enumerate}
\item If $a\in(\alpha_{0},\,\alpha_{0}^{-1})$, then $\Phi_{a}(\cdot)$
has two distinct stationary points $z_{\pm}\in\partial\mathbb{D}$
of order one, i.e.~$\frac{\partial\Phi_{a}}{\partial z}\left(z_{\pm}\right)=0$
but $\frac{\partial^{2}\Phi_{a}}{\partial z^{2}}\left(z_{\pm}\right)\neq0$,
satisfying $z_{-}=\overline{z_{+}}$. 
\item If $a\in\left\{ \alpha_{0},\alpha_{0}^{-1}\right\} $, then $\Phi_{a}(\cdot)$
has one stationary point $z_{0}\in\left\{ -1,1\right\} $ of order
two, i.e.~$\frac{\partial\Phi_{a}}{\partial z}(z_{0})=\frac{\partial^{2}\Phi_{a}}{\partial z^{2}}\left(z_{0}\right)=0$,
but $\frac{\partial^{3}\Phi_{a}}{\partial z^{3}}\left(z_{0}\right)\neq0$.
More precisely, if $a=\alpha_{0}$ then $z_{0}=-1$ and 
\[
\frac{\partial^{3}\Phi_{\alpha_{0}}}{\partial z^{3}}(z_{0})=\frac{2\lambda(1-\lambda)}{(1+\lambda)^{3}}.
\]
If $a=\alpha_{0}^{-1}$ then $z_{0}=1$ and 
\[
\frac{\partial^{3}\Phi_{\alpha_{0}^{-1}}}{\partial z^{3}}(z_{0})=\frac{2\lambda(1+\lambda)}{(1-\lambda)^{3}}.
\]
\item If $a\notin[\alpha_{0},\,\alpha_{0}^{-1}]$, then $\Phi_{a}(\cdot)$
has two stationary points $z_{\pm}\in\mathbb{R}$ of order one, i.e.
$\frac{\partial\Phi_{a}}{\partial z}\left(z_{\pm}\right)=0$ but $\frac{\partial^{2}\Phi_{a}}{\partial z^{2}}\left(z_{\pm}\right)\neq0$,
satisfying $z_{-}=z_{+}^{-1}$. 
\end{enumerate}
The stationary points $z_{+}$ and $z_{-}$ are given by the formula
\begin{equation}
z_{\pm}=z_{\pm}(a)=\frac{a(1+\lambda^{2})-(1-\lambda^{2})}{2\lambda a}\pm\sqrt{\left(\frac{a(1+\lambda^{2})-(1-\lambda^{2})}{2\lambda a}\right)^{2}-1}\label{eq:rep_z_pm}
\end{equation}
and if $a\notin\left\{ \alpha_{0},\,\alpha_{0}^{-1}\right\} $ then
\begin{eqnarray}
\frac{\partial^{2}\Phi_{a}}{\partial z^{2}}\Bigg|_{z=z_{\pm}} & = & \frac{(1-\lambda^{2})(z_{\pm}-z_{\mp})\lambda}{(z_{\pm}-\lambda)^{2}(1-\lambda z_{\pm})^{2}}.\label{eq:scd_deriv}
\end{eqnarray}
\end{lem}

\begin{proof}
Computing derivatives we obtain 
\begin{align*}
\frac{\partial\Phi_{a}}{\partial z} & =\frac{1}{z-\lambda}-\frac{a}{z}+\frac{\lambda}{1-\lambda z},\\
\frac{\partial^{2}\Phi_{a}}{\partial z^{2}} & =-\frac{1}{(z-\lambda)^{2}}+\frac{a}{z^{2}}+\frac{\lambda^{2}}{(1-\lambda z)^{2}},\\
\frac{\partial^{3}\Phi_{a}}{\partial z^{3}} & =\frac{2}{(z-\lambda)^{3}}-\frac{2a}{z^{3}}+\frac{2\lambda^{3}}{(1-\lambda z)^{3}}.
\end{align*}
The function $\Phi_{a}(z)$ has a stationary point if and only if
$\partial\Phi_{a}/\partial z=0$, i.e. if and only if 
\begin{align*}
a=1+\frac{\lambda}{z-\lambda}+\frac{\lambda z}{1-\lambda z}.
\end{align*}
Solving the latter for $z$ yields the representation \eqref{eq:rep_z_pm}
for the roots $z_{\pm}$ of $\frac{\partial\Phi_{a}}{\partial z}$.
If $a\notin\left\{ \alpha_{0},\,\alpha_{0}^{-1}\right\} ,$ then $z_{+}$
and $z_{-}$ are distinct. If $a\in(\alpha_{0},\,\alpha_{0}^{-1})$,
then $z_{\pm}\in\partial\mathbb{D}\setminus \{-1,1\}$ and if $a\notin[\alpha_{0},\,\alpha_{0}^{-1}]$,
then $z_{\pm}\in\mathbb{R}\setminus\{-1,1\}$. Plugging in the values of $z_{\pm}$ we obtain formula
\eqref{eq:scd_deriv} for the value of $\frac{\partial^{2}\Phi_{a}}{\partial z^{2}}\Big|_{z=z_{\pm}}$
when $a\notin\left\{ \alpha_{0},\,\alpha_{0}^{-1}\right\} .$ If $a\in\left\{ \alpha_{0},\,\alpha_{0}^{-1}\right\} $,
then $\frac{\partial\Phi_{a}}{\partial z}$ has a unique zero. If
$a=\alpha_{0}^{-1},$ then $z_{+}=z_{-}=1=z_{0}$ and 
\[
\Phi_{\alpha_{0}^{-1}}(1)=\frac{\partial\Phi_{\alpha_{0}^{-1}}}{\partial z}(1)=\frac{\partial^{2}\Phi_{\alpha_{0}^{-1}}}{\partial z^{2}}(1)=0,
\]
with 
\[
\frac{\partial^{3}\Phi_{\alpha_{0}^{-1}}}{\partial z^{3}}(1)=\frac{2\lambda(1+\lambda)}{(1-\lambda)^{3}}\neq0.
\]
If $a=\alpha_{0}$, then $z_{+}=z_{-}=-1=z_{0}$ and 
\[
\frac{\partial\Phi_{\alpha_{0}}}{\partial z}(-1)=\frac{\partial^{2}\Phi_{\alpha_{0}}}{\partial z^{2}}(-1)=0,\qquad\frac{\partial^{3}\Phi_{\alpha_{0}}}{\partial z^{3}}(-1)=\frac{2\lambda(1-\lambda)}{(1+\lambda)^{3}}\neq0.
\]
\end{proof}

\subsection{Proof of Theorem \ref{Th_Regions_I_VII}~(1),~(2)}
\begin{proof} We start with part (1). 
We establish by induction on $k$ that 
\[
(b_{\lambda}^{n})^{(k)}(0)\sim(-\lambda)^{n-k}\left(n(1-\lambda^{2})\right)^{k},\qquad k\geq0,\,n\to\infty.
\]
This asymptotic formula clearly holds for $k=0$. We assume that the
above induction hypothesis holds for all $0\leq j\leq k$. We first
observe that 
$$
(b_{\lambda}^{n})^{(k+1)}(z)  =\left((b_{\lambda}^{n})'\right)^{(k)}(z)=n(1-\lambda^{2})\left((1-\lambda z)^{-2}\cdot b_{\lambda}^{n-1}\right)^{(k)}(z),
$$
and then apply Leibniz formula to the product $z\mapsto(1-\lambda z)^{-2}\cdot b_{\lambda}^{n-1}(z)$,
at $z=0$. Computation shows that 
\[
\left((1-\lambda z)^{-2}\right)^{(j)}(0)=(j+1)!\lambda^{j}
\]
and therefore 
\[
\left((1-\lambda z)^{-2}\cdot b_{\lambda}^{n-1}\right)^{(k)}(0)=\sum_{j=0}^{k}\binom{k}{j}(j+1)!\lambda^{j}\left(b_{\lambda}^{n-1}\right)^{(k-j)}(0).
\]
Applying our induction hypothesis to the factor $\left(b_{\lambda}^{n-1}\right)^{(k-j)}(0)$,
it turns out that the main contribution to the above sum is due to
its first term (whose index is $j=0$): 
$$
\binom{k}{0}(0+1)!\lambda^{0}\left(b_{\lambda}^{n-1}\right)^{(k)}(0)  \sim(-\lambda)^{n-1-k}\left((n-1)(1-\lambda^{2})\right)^{k} \sim(-\lambda)^{n-k-1}\left(n(1-\lambda^{2})\right)^{k}.
$$
We conclude that 
$$
(b_{\lambda}^{n})^{(k+1)}(0)  =n(1-\lambda^{2})\left((1-\lambda z)^{-2}\cdot b_{\lambda}^{n-1}\right)^{(k)}(0)
  \sim(-\lambda)^{n-k-1}\left(n(1-\lambda^{2})\right)^{k+1},
$$
which completes the proof of part (1). \\

Proof of part (2). The integral defining $\widehat{b_{\lambda}^{n}}(k)$
is of the form: 
\begin{equation}
\widehat{b_{\lambda}^{n}}(k)=\frac{1}{2i\pi}\int_{\partial\mathbb{D}}\varphi(z)e^{n\Phi(z)}{\rm d}z\label{eq:int_rep}
\end{equation}
where $\varphi(z)=z^{-1}$ and $\Phi=\Phi_{a}$ with $a=k/n$.

\textit{\uline{Case 1}}\textit{:}\textbf{\textit{ $a\in[\epsilon,\alpha].$
}}We first assume that $a\in[\epsilon,\alpha]$ for a given $\epsilon\in(0,\alpha)$
and apply the saddle point/steepest descent method \cite[Chapter 7]{BlHa},
\cite[Chapters 5-6]{Bru}, \cite[Chapters 7-8]{Cop} to determine
an asymptotic formula for the integral \eqref{eq:int_rep}. This method
essentially consists in deforming the original contour of integration
(here $\partial\mathbb{D}$) into a suitable one, say $C,$ so that
$C$ remains inside the domain $U$ where our integrand is holomorphic
(here $U=\mathbb{C}\setminus\{1/\lambda\}$) and the classical conditions
-- which we recall below and which relate to geometrical considerations
specific to our situation -- are satisfied. We refer to Figure \ref{Plot_1}
and Figure \ref{Plot_2} for an illustration.

1) First of all \textit{$C$ must pass through the relevant saddle
point(s) of} $\Phi$ i.e. the solutions $z_{\pm}$ of the equation
$\Phi'(z)=0$. In our case $a\leq\alpha<\alpha_{0}$ it can be checked
that only $z_{+}$ is relevant: on the interval $[z_{-},\,z_{+}]$
the continuous function 
\[
\psi:z\mapsto e^{\Re\Phi(z)}
\]
achieves its minimum at $z=z_{+},$ its maximum at $z=z_{-}$ and
\[
\psi(z_{+})<1<\psi(z_{-}).
\]

We also observe that the function $a\mapsto z_{+}(a)$ is negative
and monotonically decreasing for $a\in(0,\alpha_{0})$; moreover $\lim_{a\rightarrow0}z_{+}(a)=0$
and $\lim_{a\rightarrow\alpha_{0}}z_{+}(a)=-1$. In particular for
$a\in[\epsilon,\alpha]$ we have $-1<z_{+}(a)<0$ and $z_{+}(a)$
is separated from 0.

2) The level curve 
\[
L(z_{+})=\left\{ z\in U:\:\:\Re\Phi(z)=\Re\Phi(z_{+})\right\} 
\]
passes two times through $z_{+}$ making angle of $\pi/2$ and divides
$U$ into two domains $V(z_{+})$ and $H(z_{+})$ respectively named
valleys and hills separating the neighborhood of the saddle point
$z_{+}$: 
\[
V(z_{+})=\left\{ z\in U:\:\:\Re\Phi(z)<\Re\Phi(z_{+})\right\} ,
\]
\[
H(z_{+})=\left\{ z\in U:\:\:\Re\Phi(z)>\Re\Phi(z_{+})\right\} ,
\]
and the new contour of integration $C$\textit{ must be contained
in }$V(z_{+}).$ Here we observe that $L(z_{+})$ is symmetric with
respect to the real axis, which is the bisector in $H(z_{+})$ of
the angle between the two tangents to the curve $L(z_{+})$ at $z_{+}.$
We have $\psi(z)=1$ for $z\in\partial\mathbb{D}$ and therefore $\partial\mathbb{D}\subset H(z_{+}).$
Furthermore we observe that $H(z_{+})$ contains both a neighborhood
of $1/\lambda$ because $\lim_{z\to1/\lambda}\psi(z)=\infty$,
and a neighborhood of $0$ since $\lim_{z\to0}\psi(z)=\infty.$
The fact that $\lim_{z\to\infty}\psi(z)=0$ shows that $V(z_{+})$
contains a neighborhood of $\infty$ and that the distance from any
point of $L(z_{+})$ to $z_{+}$ is finite. $V(z_{+})$ also contains
a neighborhood of $\lambda$ since $\psi(\lambda)=0.$ Let us finally
mention that $L(z_{+})$ is actually composed of two curves: A closed
curve contained in $\mathbb{D}$ passing two times through $z_{+}$
and another one surrounding $\partial\mathbb{D},$ which is not of
interest for us. We refer to Figure \ref{Plot_1} for a depiction
of the behavior of $L(z_{+}),\,H(z_{+})$ and $V(z_{+})$ near the
unit disc.

3) We recall that the curves of steepest descent respectively steepest
ascent from $z_{+}$, respectively named $S_{d}$ and $S_{a}$, see
Figure \ref{Plot_2}, are the curves defined by the equation 
\[
\Im\Phi(z)=\Im\Phi(z_{+})
\]
and contained in $V(z_{+})$ -- respectively in $H(z_{+})$ -- and
in a neighborhood of $z_{+}$. If $T(z_{+})$ denotes the tangent
at $z_{+}$ to the curve of steepest descent from $z_{+},$ then\textit{
$T(z_{+})$ must also be tangent to the new contour of integration
$C$ at $z_{+}$} and it is more convenient to choose $C$ such that
it coincides with $T(z_{+})$ on a fixed neighborhood of $z_{+}$.
Here $T(z_{+})$ is the vertical line passing through $z_{+}$. It
is usually obtained as the bisector in $V(z_{+})$ of the angle formed
by the two perpendicular tangents to the level curve $L(z_{+})$ at
$z_{+}.$ The other bisector of this angle is part of the real axis,
and necessarily lies in $H(z_{+}):$ $z\mapsto\psi(z)$ achieves its
minimum at $z_{+}$ on $[z_{-},\,0)$ whereas $z\mapsto\psi(z)$ attains
its maximum at $z_{+}$ on $T(z_{+})$, which is required to apply
the method of the steepest descent.

\begin{figure}[h]
\centering \centering \includegraphics[width=0.9\linewidth]{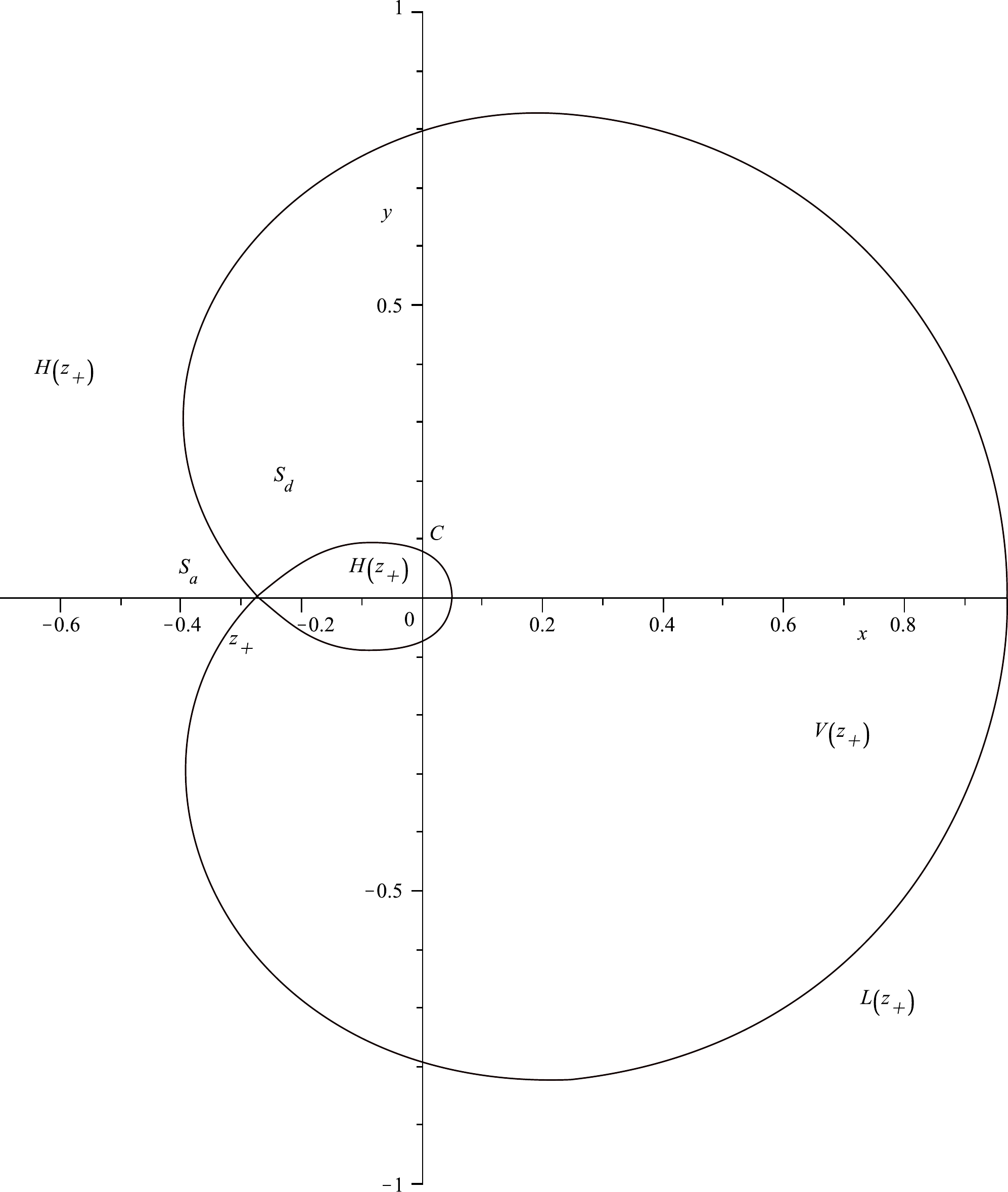}
\caption{This figure depicts $L(z_{+})$ $\,H(z_{+})$ and $V(z_{+})$
where $\lambda=0.5$ and $k/n=0.32$. }
\label{Plot_1} 
\end{figure}

\begin{figure}[h!]
\centering \centering \includegraphics[width=0.9\linewidth]{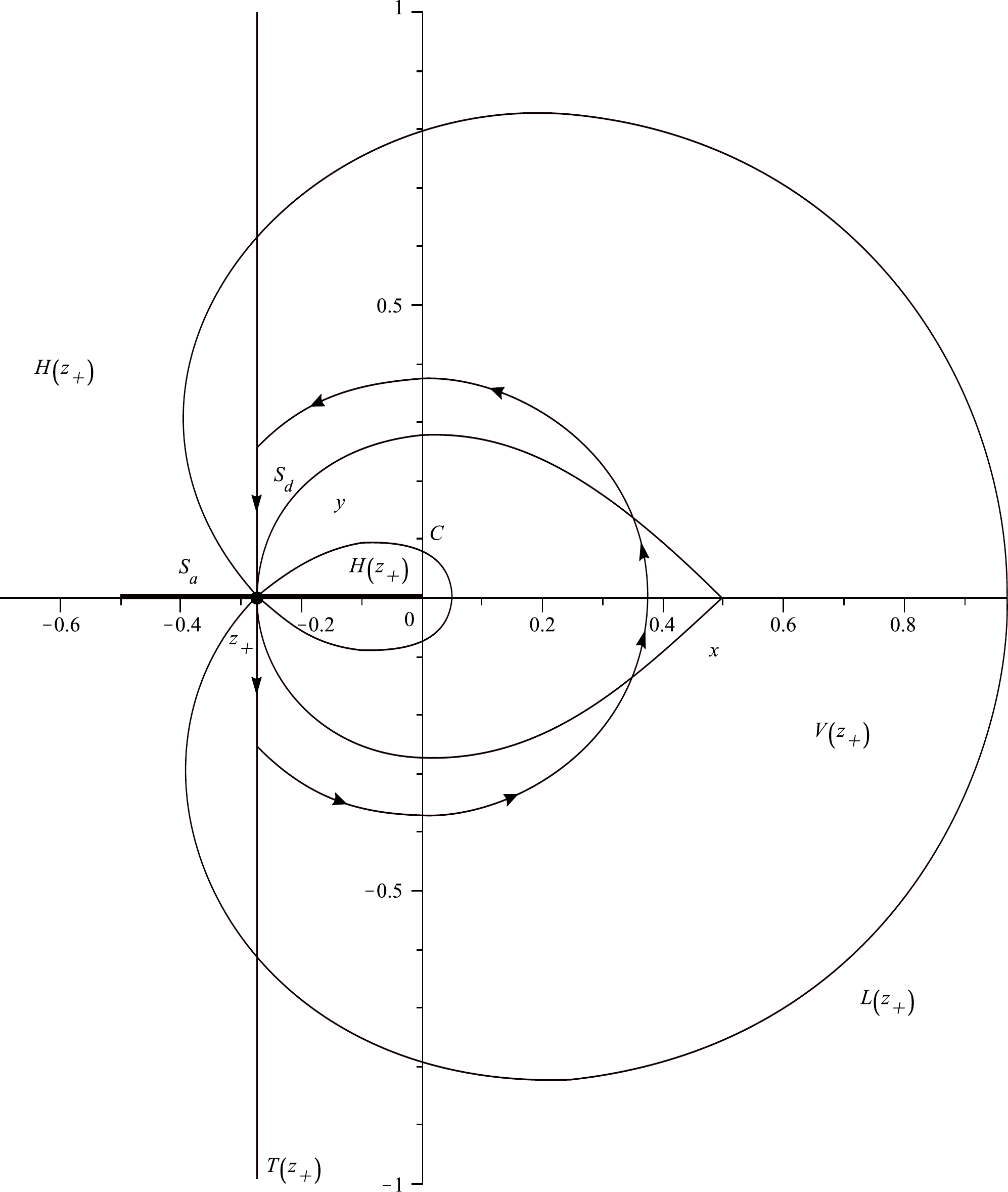}
\caption{This figure depicts the new contour of integration $C$, the level
curve $L(z_{+})$, the curve $S_{d}$ of steepest descent from $z_{+}$,
the curve $S_{a}$ of steepest ascent from $z_{+}$, the tangent $T(z_{+})$
to $S_{d}$ at $z_{+}$, the domain $V(z_{+})$ and the domain $H(z_{+})$,
when $k/n\in[\epsilon n,\alpha n]$. Here we chose $\lambda=0.5$
and $k/n=0.32$. }
\label{Plot_2} 
\end{figure}

If such a choice of $C$ is possible -- which is the case here, see
Figure \ref{Plot_2} -- then \cite[formula (7.2.10)]{BlHa}, \cite[formula (36.7)]{Cop},
\cite[formula (5.7.2)]{Bru} we have 
\begin{align}
\int_{\partial\mathbb{D}}\varphi(z)e^{n\Phi(z)}{\rm d}z & =\int_{C}\varphi(z)e^{n\Phi(z)}{\rm d}z\nonumber \\
 & \sim\varphi(z_{+})e^{n\Phi(z_{+})+i\theta}\sqrt{\frac{2\pi}{n\abs{\Phi''(z_{+})}}},\qquad n\rightarrow\infty,\label{eq:asympt_formula}
\end{align}
where $\theta$ is the angle between $T(z_{+})$ and the real axis.
It follows from Lemma \ref{lem:critical_pts_f}, formula \eqref{eq:scd_deriv},
that 
\begin{equation}
\Phi''(z_{+})=\frac{\lambda(z_{+}-z_{-})(1-\lambda^{2})}{(z_{+}-\lambda)^{2}(1-\lambda z_{+})^{2}},\label{eq:sec_der_Phi_z_p}
\end{equation}
which is strictly positive, and taking into account the fact that
$\theta=3\pi/2$ we find 
\begin{align*}
\int_{\partial\mathbb{D}}\varphi(z)e^{n\Phi(z)}{\rm d}z & \sim i\sqrt{\frac{2\pi}{n}}\frac{1}{z_{+}}\left(\frac{b_{\lambda}(z_{+})}{z_{+}^{k/n}}\right)^{n}\frac{(z_{+}-\lambda)(1-\lambda z_{+})}{\sqrt{\lambda(1-\lambda^{2})(z_{+}-z_{-})}}\\
 & =i\sqrt{\frac{2\pi}{n}}\left(\frac{b_{\lambda}(z_{+})}{z_{+}^{k/n}}\right)^{n}\frac{(z_{+}-\lambda)(z_{-}-\lambda)}{\sqrt{\lambda(1-\lambda^{2})(z_{+}-z_{-})}},
\end{align*}
where we used the identity $z_{+}z_{-}=1$ (see Lemma \ref{lem:critical_pts_f}).
It follows from \eqref{eq:rep_z_pm} that 
\begin{align}
z_{+}-z_{-} & =\frac{\sqrt{(\lambda^{2}-1)(a(\lambda-1)+1+\lambda)(a(1+\lambda)+\lambda-1)}}{a\lambda}\nonumber \\
 & =\frac{1-\lambda^{2}}{a\lambda}\sqrt{(a-\alpha_{0}^{-1})(a-\alpha_{0})},\label{eq:z_p_minus_z_m}
\end{align}
and that 
\begin{equation}
(z_{+}-\lambda)(z_{-}-\lambda)=\frac{1-\lambda^{2}}{a},\label{eq:zp-r_times_zm-r}
\end{equation}
where $a=k/n.$ Therefore 
\begin{align*}
\frac{(z_{+}-\lambda)(z_{-}-\lambda)}{\sqrt{\lambda(1-\lambda^{2})(z_{+}-z_{-})}} & =\frac{1-\lambda^{2}}{a}\frac{1}{\sqrt{\lambda(1-\lambda^{2})}}\sqrt{\frac{a\lambda}{1-\lambda^{2}}}\frac{1}{\left[(a-\alpha_{0}^{-1})(a-\alpha_{0})\right]^{1/4}}\\
 & =\frac{1}{\sqrt{a}\left[(a-\alpha_{0}^{-1})(a-\alpha_{0})\right]^{1/4}}.
\end{align*}
Dividing the above asymptotic formula for $\int_{\partial\mathbb{D}}\varphi(z)e^{n\Phi(z)}{\rm d}z$
by $2i\pi$ we conclude that 
\[
\widehat{b_{\lambda}^{n}}(k)\sim\frac{1}{\sqrt{2n\pi}}\frac{1}{\sqrt{k/n}\left[(\alpha_{0}-k/n)(\alpha_{0}^{-1}-k/n)\right]^{1/4}}\left(\frac{b_{\lambda}(z_{+})}{z_{+}^{k/n}}\right)^{n}.
\]
\\

\textit{\uline{Case 2}}\textit{:}\textbf{\textit{ }}$a=k/n\rightarrow0$
\textit{and} $k\rightarrow\infty$\textbf{\textit{.}} Now we assume
that $k=k(n)$ is such that $k(n)\rightarrow\infty$ and $k(n)/n\rightarrow0$
as $n\rightarrow\infty.$ The situation is essentially the same as
before in the sense that again $z_{+}=z_{+}(k/n)$ is the only relevant
saddle point of $\Phi$, but it is slightly more delicate because
this time $z_{+}$ approaches the origin as $n\rightarrow\infty.$
The new contour of integration $C$ is chosen in\textit{ }$V(z_{+})$
the same way as previously but the straight steepest descent line
$C\cap T(z_{+})$ -- along which $\Phi''(z_{+})(z-z_{+})^{2}$ is
negative -- must lie in a neighborhood of $z_{+}$ where $\Phi$
can be expanded as a convergent power series 
\[
\Phi(z)=\Phi(z_{+})+\sum_{j\geq2}\frac{\Phi^{(j)}(z_{+})}{j!}(z-z_{+})^{j}.
\]
A computation shows that 
\[
z_{+}=-a\frac{\lambda}{1-\lambda^{2}}+\cO(a^{2})
\]
as $a=k/n$ tends to 0, and for $j\geq2$ 
\begin{align}
\frac{\Phi^{(j)}(z_{+})}{j!} & =\frac{(-1)^{j}a}{jz_{+}^{j}}+\frac{\lambda^{j}}{j}\frac{1}{(1-\lambda z_{+})^{j}}-\frac{(-1)^{j}}{j(z_{+}-\lambda)^{j}}\label{eq:Taylor_coeff_Phi}\\
 & \sim\frac{(-1)^{j}a}{jz_{+}^{j}}\sim\frac{1}{ja^{j-1}}\left(\frac{1-\lambda^{2}}{\lambda}\right)^{j}.\nonumber 
\end{align}
In particular, for large enough $n$ the radius of convergence $R$
of the power series of $\Phi$ near $z_{+}$ is proportional to $a.$
We put 
\[
G(z)=\sum_{j\geq3}\frac{\Phi^{(j)}(z_{+})}{j!}(z-z_{+})^{j}.
\]
We follow and adapt the approach from \cite[p. 92-93]{Cop} to our
situation. Let $x=2/5$ and $u_{k}=k^{-x}$ so that $\lim_{k\rightarrow\infty}u_{k}=0,$
$\lim_{k\rightarrow\infty}ku_{k}^{3}=0$ and $\lim_{k\rightarrow\infty}ku_{k}^{2}=\infty$.
We choose $C$ such that $C\cap T(z_{+})$ lies in the disc $\abs{z-z_{+}}\leq\rho$
where $\rho=au_{k}=\frac{k}{n}u_{k}.$ For $z$ in the disc $\abs{z-z_{+}}\leq\rho$
we have 
\begin{align*}
\abs{G(z)} & \leq\sum_{j\geq3}\frac{\abs{\Phi^{(j)}(z_{+})}}{j!}\abs{z-z_{+}}^{j}\\
 & \lesssim a\sum_{j\geq3}\frac{1}{j}\left(\frac{1-\lambda^{2}}{\lambda}u_{k}\right)^{j}\lesssim au_{k}^{3}.
\end{align*}
It follows that for $z\in C\cap T(z_{+})$ we have 
\[
\exp\left(n\Phi(z)\right)=\exp\left(n\Phi(z_{+})+n\frac{\Phi''(z_{+})}{2}(z-z_{+})^{2}\right)\cdot\left(1+\cO\left(ku_{k}^{3}\right)\right).
\]
Observing that $\Abs{\frac{z_{+}-z}{z_{+}}}\lesssim u_{k}$, we obtain 
$$
\varphi(z)  =\frac{1}{z_{+}+z-z_{+}}=\frac{1}{z_{+}\left(1+\frac{z-z_{+}}{z_{+}}\right)}
  =\frac{1}{z_{+}}\left(1+\sum_{j\geq1}\frac{1}{z_{+}}\left(\frac{z_{+}-z}{z_{+}}\right)^{j}\right)=\frac{1}{z_{+}}+\cO\left(u_{k}\right).
$$
Taking into account the fact that $x<2^{-1}$ we find that for $z\in C\cap T(z_{+})$
\[
\varphi(z)\exp\left(n\Phi(z)\right)=z_{+}^{-1}\exp\left(n\Phi(z_{+})+n\frac{\Phi''(z_{+})}{2}(z-z_{+})^{2}\right)\cdot\left(1+\cO\left(ku_{k}^{3}\right)\right)
\]
The contribution of the neighbourhood $\abs{z-z_{+}}\leq\rho$ of
the saddle point $z_{+}$ is therefore 
\begin{multline}
\int_{C\cap T(z_{+})}\varphi(z)\exp(n\Phi(z))\d z\\
=z_{+}^{-1}\exp\left(n\Phi(z_{+})\right)\int_{C\cap T(z_{+})}\exp\left(n\frac{\Phi''(z_{+})}{2}(z-z_{+})^{2}\right)\d z\cdot\left(1+\cO\left(ku_{k}^{3}\right)\right)\label{eq:copson}
\end{multline}
It follows from \eqref{eq:sec_der_Phi_z_p}, \eqref{eq:z_p_minus_z_m},
and \eqref{eq:zp-r_times_zm-r} that 
\begin{equation}
\Phi''(z_{+})=\frac{k}{n}z_{+}^{-2}\left(\alpha_{0}-\frac{k}{n}\right)^{1/2}\left(\alpha_{0}^{-1}-\frac{k}{n}\right)^{1/2}.\label{eq:sec_der_Phi_z_p_2}
\end{equation}
We let $r$ vary from $-\rho$ to $\rho$ and put $z=z_{+}-ir$. Then
\eqref{eq:copson} gives 
\begin{multline*}
\int_{C\cap T(z_{+})}\varphi(z)\exp(n\Phi(z))\d z\\
=-iz_{+}^{-1}\left(1+\cO\left(ku_{k}^{3}\right)\right)\exp\left(n\Phi(z_{+})\right)\int_{-\rho}^{\rho}\exp\left(-kz_{+}^{-2}\frac{\left(\alpha_{0}-\frac{k}{n}\right)^{1/2}\left(\alpha_{0}^{-1}-\frac{k}{n}\right)^{1/2}}{2}r^{2}\right)\d r.
\end{multline*}
Changing the variable $r$ by 
\[
v=\sqrt{k}z_{+}^{-1}\frac{\left(\alpha_{0}-\frac{k}{n}\right)^{1/4}\left(\alpha_{0}^{-1}-\frac{k}{n}\right)^{1/4}}{\sqrt{2}}r,
\]
we get 
\begin{multline}
\int_{C\cap T(z_{+})}\varphi(z)\exp(n\Phi(z))\d z\\
=i\left(1+o(1)\right)\exp\left(n\Phi(z_{+})\right)\sqrt{\frac{2}{k}}\int_{-\omega}^{\omega}\exp\left(-v^{2}\right)\d v,\label{eq:Copson2}
\end{multline}
where 
$$
\omega  =\sqrt{\frac{k}{2}}\abs{z_{+}}^{-1}\left(\alpha_{0}-\frac{k}{n}\right)^{1/4}\left(\alpha_{0}^{-1}-\frac{k}{n}\right)^{1/4}\rho
 \sim\sqrt{\frac{k}{2}}\frac{1-\lambda^{2}}{\lambda}u_{k}\asymp\sqrt{k}u_{k},
$$
and, in particular, $\omega$ tends to $\infty$ with $k$. Moreover, 
as $k\rightarrow\infty$, we have 
$$
\int_{\omega}^{\infty}\exp(-v^{2})\d v  =\cO\left(\frac{\exp(-\omega^{2})}{\omega}\right)
  =\cO\left(\frac{\exp(-Cku_{k}^{2})}{\sqrt{k}u_{k}}\right),
$$
for some absolute constant $C>0$. Therefore,  
\[
\int_{C\cap T(z_{+})}\varphi(z)\exp(n\Phi(z))\d z=i\exp\left(n\Phi(z_{+})\right)\sqrt{\frac{2\pi}{k}}\cdot\left(1+o(1)\right),
\]
and, hence,  
\[
\frac{1}{2i\pi}\int_{C\cap T(z_{+})}\varphi(z)\exp(n\Phi(z))\d z\sim\frac{1}{\sqrt{2k\pi}}\left(\frac{b_{\lambda}(z_{+})}{z_{+}^{k/n}}\right)^{n}.
\]
\\
 To complete the proof we choose $C$ so that $C\setminus T(z_{+})$
coincides with the circle centered at $0$ of radius $\abs{z_{+}}=-z_{+}$
intersected with the half-plane $\left\{ \Re z>z_{+}\right\} $, and
show that 
\[
\frac{1}{2i\pi}\int_{C\setminus T(z_{+})}\varphi(z)\exp(n\Phi(z))\d z=o\left(\frac{1}{2i\pi}\int_{C\cap T(z_{+})}\varphi(z)\exp(n\Phi(z))\d z\right)
\]
as $k\rightarrow\infty$. The endpoints of $C\setminus T(z_{+})$
are denoted by $\abs{z_{+}}e^{i(\pi-\eta)}$ and $\abs{z_{+}}e^{i(-\pi+\eta)}$
where $\eta>0$ is such that 
\[
\eta\asymp\sin\eta\asymp\frac{\rho}{\abs{z_{+}}}\asymp u_{k}.
\]
We write 
\begin{multline*}
\frac{1}{2i\pi}\int_{C\setminus T(z_{+})}\varphi(z)\exp(n\Phi(z))\d z\\
=\varphi(z_{+})\exp(n\Phi(z_{+}))\frac{1}{2i\pi}\int_{C\setminus T(z_{+})}\frac{\varphi(z)\exp(n\Phi(z))}{\varphi(z_{+})\exp(n\Phi(z_{+}))}\d z,
\end{multline*}
put $z=\abs{z_{+}}e^{it}=-z_{+}e^{it},$ and observe that 
\[
\Abs{\frac{\varphi(z)\exp(n\Phi(z))}{\varphi(z_{+})\exp(n\Phi(z_{+}))}}=\Abs{\frac{b_{\lambda}\left(\abs{z_{+}}e^{it}\right)}{b_{\lambda}(z_{+})}}^{n},\qquad\abs{z}=\abs{z_{+}}.
\]
A direct computation shows that 
\begin{equation}
\Abs{b_{\lambda}\left(\abs{z_{+}}e^{it}\right)}^{2}=1-\frac{(1-\lambda^{2})(1-\abs{z_{+}}^{2})}{1+\lambda^{2}\abs{z_{+}}^{2}-2\lambda\abs{z_{+}}\cos t}.\label{eq:b_square_identity}
\end{equation}
This function is increasing on $[0,\pi]$ and decreasing on $[-\pi,0]$.
Therefore,  
\[
\Abs{\frac{\varphi(z)\exp(n\Phi(z))}{\varphi(z_{+})\exp(n\Phi(z_{+}))}}\leq\Abs{\frac{b_{\lambda}\left(\abs{z_{+}}e^{i(\pi-\eta)}\right)}{b_{\lambda}(z_{+})}}^{n}.
\]
By \eqref{eq:b_square_identity} we obtain that there exists $C>0$
such that 
\[
\Abs{\frac{b_{\lambda}\left(\abs{z_{+}}e^{i(\pi-\eta)}\right)}{b_{\lambda}(z_{+})}}\leq1-Ca\eta^{2},
\]
which proves that 
\begin{equation}
\frac{1}{2i\pi}\int_{C\setminus T(z_{+})}\frac{\varphi(z)\exp(n\Phi(z))}{\varphi(z_{+})\exp(n\Phi(z_{+}))}\d z=\cO\left(\exp(-Cku_{k}^{2})\right).\label{eq:remainder_SDM}
\end{equation}
This completes the proof in case 2.

\textit{\uline{Case 3}}\textit{: }\textbf{\textit{$a\in[\alpha^{-1},\epsilon^{-1}]$.}}
A discussion similar to that for case 1 leads to the same formula
for $a=k/n\in[\alpha^{-1},\epsilon^{-1}]$ where $\epsilon\in(0,\alpha)$
is fixed. We first reproduce the three steps from the first case required
to deform the original contour of integration $\partial\mathbb{D}$
into the suitable one $C,$ which remains inside the domain $U$ where
our integrand is holomorphic. The geometrical considerations corresponding
to conditions (1)--(3) are sometimes slightly different in this case.
We detail them below for completeness and refer to Figure \ref{Plot_3}
for an illustration.

1) As in case 1, $C$ should pass through the relevant saddle point
of $\Phi$. Again, it can be checked that only the critical point
$z_{+}$ is relevant: For $z$ on the interval $[\lambda,\lambda^{-1})$
the continuous function $z\mapsto\psi(z)$ achieves its minimum at
$z=z_{+},$ its maximum at $z=z_{-}$ and 
\[
\psi(z_{+})<1<\psi(z_{-}).
\]

We also observe that the function $a\mapsto z_{+}(a)$ is nonnegative
and monotonically increasing for $a\in(\alpha_{0}^{-1},\epsilon^{-1})$;
moreover $\lim_{a\rightarrow\alpha_{0}^{-1}}z_{+}(a)=1$ and $\lim_{a\rightarrow+\infty}z_{+}(a)=1/\lambda$.
In particular for $a\in[\epsilon,\alpha]$ we have $1<z_{+}(a)<1/\lambda$.

2) Again, the level curve $L(z_{+})$ passes two times through $z_{+}$
making angle of $\pi/2$ and divides $U$ into $V(z_{+})$ (valleys)
and $H(z_{+})$ (hills). The new contour of integration $C$ will
be contained in $V(z_{+})$ as required. $L(z_{+})$ is symmetric
with respect to the real axis and it consists again of two parts.
The first one is not of interest for us: it is a closed curve contained
in $\mathbb{D}$ surrounding $\lambda.$ The second one, which is
the one we are interested in, is a closed curve that surrounds $\partial\mathbb{D}$
to the left of $z_{+}$ and a neighborhood of $1/\lambda$ to the
right of $z_{+}.$ As in case 1, the real axis is the bisector in
$H(z_{+})$ of the angle between the two tangents to this part of
$L(z_{+})$ at $z_{+}.$ Finally $H(z_{+})$ still contains $\partial\mathbb{D}$
since $\psi(z)=1$ for $z\in\partial\mathbb{D},$ and it also contains
a neighborhood of $1/\lambda$ because $\lim_{z\mapsto1/\lambda}\psi(z)=\infty$.
$V(z_{+})$ contains a neighborhood of $\infty$ because $\lim_{z\mapsto\infty}\psi(z)=0$
and also contains a neighborhood of $\lambda$ because $\lim_{z\mapsto\lambda}\psi(z)=0.$

3) We do not reproduce the discussion on the curves of steepest descent/ascent
$S_{d}$ and $S_{a}$ from $z_{+}$, since it is identical to the
previous one (case 1). (This time $z\mapsto\psi(z)$ attains its
minimum on $(\lambda,1/\lambda)$ at $z_{+}$ whereas $z\mapsto\psi(z)$
attains its maximum at $z_{+}$ on $T(z_{+})$.)

\begin{figure}[h]
\centering \centering \includegraphics[width=0.9\linewidth]{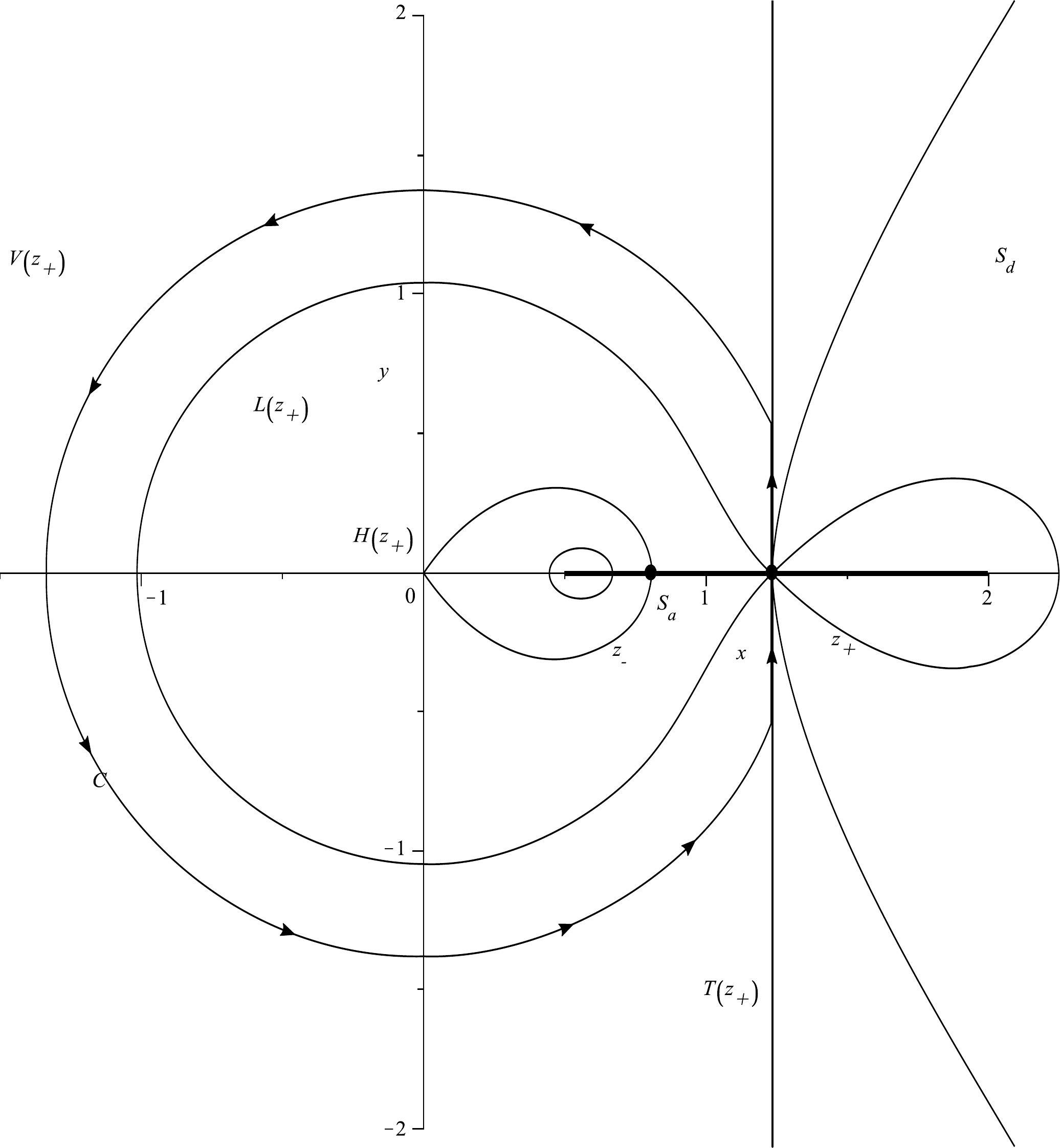}
\caption{This figure depicts the new contour of integration $C$, the level
curve $L(z_{+})$, the curve $S_{d}$ of steepest descent from $z_{+}$,
the curve $S_{a}$ of steepest ascent from $z_{+}$, the tangent $T(z_{+})$
to $S_{d}$ at $z_{+}$, the domain $V(z_{+})$ and the domain $H(z_{+})$,
when $k/n\in[\alpha^{-1}n,\epsilon^{-1}n]$. Here we chose $\lambda=0.5$
and $k/n=3.3$. }
\label{Plot_3} 
\end{figure}

Since such a choice of $C$ is possible -- see Figure \ref{Plot_3}
for an illustration -- the asymptotic formula \eqref{eq:asympt_formula}
used in case 1 applies also here and we get 
\[
\widehat{b_{\lambda}^{n}}(k)=\frac{1}{2i\pi}\int_{C}\varphi(z)e^{n\Phi(z)}{\rm d}z\sim\frac{1}{2i\pi}\varphi(z_{+})e^{n\Phi(z_{+})+i\theta}\sqrt{\frac{2\pi}{n\abs{\Phi''(z_{+})}}},
\]
as $n\rightarrow\infty$, where $\theta=\pi/2$ is the angle between
$T(z_{+})$ and the real axis. The rest of the proof is identical
to the one we have detailed in case 1.

\textit{\uline{Case }}\textit{4:}\textbf{\textit{ }}$a=k/n\in[\alpha^{-1}n,\infty)$
and $k/n\rightarrow\infty$. This case is analogous to case 2. As
in case 3, $z_{+}=z_{+}(k/n)$ is the only relevant saddle point of
$\Phi$, but this time $z_{+}$ approaches $1/\lambda$ as $n\rightarrow\infty.$
The new contour of integration $C$ is chosen in\textit{ }$V(z_{+})$
the same way as in case 3 but the straight steepest descent line $C\cap T(z_{+})$
-- along which $\Phi''(z_{+})(z-z_{+})^{2}$ is negative -- must
lie (as in case 2) in a neighborhood of $z_{+}$ where $\Phi$ can
be expanded as a convergent power series 
\[
\Phi(z)=\Phi(z_{+})+\sum_{j\geq2}\frac{\Phi^{(j)}(z_{+})}{j!}(z-z_{+})^{j}
\]
whose radius of convergence -- which can be computed using \eqref{eq:Taylor_coeff_Phi}
-- is this time proportional to $1/a$, whereas it was proportional
to $a$ when $k/n\rightarrow0$ (see case 2). We omit the rest of
the proof, which is identical to the one we have detailed in case
2. 
\end{proof}

\subsection{Proof of Proposition \ref{Prop_Airy_fcts}}
\begin{proof}[Proof of Proposition \ref{Prop_Airy_fcts}.]
We omit the proof of the second asymptotic formula (i.e. when $k/n$
is in a neighborhood of $\alpha_{0}^{-1}$) because it follows from
an almost word-for-word adaptation of the one of \cite[Proposition 10]{SzZa2}
(the part corresponding to (2)--(4), replacing the factor $(1-z^{-2})$
by 1). We choose to sketch the proof of the asymptotic formulas for
$k/n$ in a neighborhood of $\alpha_{0}$ , which is similar to those
in \cite{SzZa2,SzZa4}, but where computations are slighlty different.
We refer to the proof of \cite[Proposition 17]{SzZa4} for more technical
details. Again, we recall that for any $k$ and $n$: 
\[
\widehat{b_{\lambda}^{n}}(k)=\frac{1}{2i\pi}\int_{\partial\mathbb{D}}\frac{e^{n\Phi(z)}}{z}\d z
\]
where $\Phi=\Phi_{a}$ and $a=k/n$. It is explained in \cite{SzZa2,SzZa1}
that the standard method of stationary phase cannot be applied when
$k/n$ approaches $\alpha_{0}^{-1}$ because in this case the saddle
points $z_{+}$ and $z_{-}$ which are of order 1, are coalescing
to the saddle point $z_{0}=1,$ which is of order 2. If $k/n$ approaches
$\alpha_{0}$, then the same phenomena occurs and $z_{\pm}$ are coalescing
this time to $z_{0}=-1.$ As the main contribution of the above integral
is due to the critical points $z_{\pm}=z_{\pm}(a)$ of $\Phi_{a}$,
if $a<\alpha_{0}$ it is required to locally deform the unit circle
to a new contour that passes through $z_{+}$, $z_{-}$ (which are
real and negative) and $-1$. If $a>\alpha_{0}$, then the critical
points $z_{\pm}\in\partial\mathbb{D}$ (are complex conjugates) and
there is no need to deform the contour as the unit circle already
passes through $z_{+}$, $z_{-}$ and $-1$: In this case the proof
below is actually reduced to an application of the uniform version
of the method of stationary phase \cite[Section 2.3]{Bor}. Let $\mathcal{D}(-1,\varepsilon)$
be the closed disk centered at $-1$ of radius $\varepsilon>0$ chosen
in such a way that $z_{\pm}\in\mathcal{D}(-1,\varepsilon)$. We denote
by $\mathcal{C}_{\epsilon}\subset D(-1,\varepsilon)$ a corresponding
local deformation of the unit circle $\partial\mathbb{D}$ and illustrate
it below. \vskip 5pt 
\begin{figure}[H]
\begin{minipage}[c]{0.43\textwidth}%
\hskip 0.8cm \begin{tikzpicture}[scale=1.8]
\scalebox{-1}[1]{\draw (11,1) [dashed] arc (90:15:1);
\draw [dashed] (11,1) arc (90:0:1);
\draw [dashed] (11,-1) arc (270:360:1);
\draw [dashed] (11,-1) arc (270:345:1);

\draw [color=blue,very thick] (11.96,0.27) arc (16:20:1);
\draw [color=blue,very thick] (11.96,0.27) to[out=-90,in=90](12.2,0);
\draw [color=blue,very thick](12.2,0) to (11.8,0);
\draw [color=blue,very thick] (11.8,0) to[out=-90, in=90](11.96,-0.27);
\draw [color=blue,very thick] (11.945,-0.33) arc (346:350:1);}

\draw [color=blue] (-12.1,0.3) node [scale=1] {$\mathcal{C}_\varepsilon$};
\draw (-11.8,-0.05)--(-11.8,0.05);
\draw (-12,-0.05) -- (-12,0.05);
\draw (-12.2,-0.05)--(-12.2,0.05);
\draw (-12.1,-0.1) node[left,scale=1] {$z_-$};
\draw (-11.85,-0.1) node[right,scale=1]{$z_+$};

\draw (-12,-0.19) node[above,scale=0.6] {$-1$};
\end{tikzpicture} \caption{The contour $\mathcal{C}_{\epsilon}$ for $\frac{k}{n}<\alpha_{0}$.}
\end{minipage}\hfill{}%
\begin{minipage}[c]{0.43\textwidth}%
\hskip 1cm \begin{tikzpicture}[scale=1.8]

\draw [dashed] (-11,1) arc (90:180:1);
\draw [dashed] (-11,-1) arc (270:180:1);

\draw[color = blue, very thick] (-11.9,0.42) arc (150:182:0.8) ;
\draw[color = blue, very thick] (-11.9,-0.42) arc (210:178:0.8) ;
\draw [color=blue] (-12.1,0.3) node [scale=1] {$\mathcal{C}_\varepsilon$};

\draw (-12,0.2) node[right] {$z_+$};
\draw (-12,-0.2) node[right]{$z_-$};
\draw (-12.03, 0.2) -- (-11.93,0.2);
\draw (-12.03, -0.2) -- (-11.93,-0.2);
\draw (-12.13,0.1) node[below,scale=0.6] {$-1$};
\end{tikzpicture}

\caption{The contour $\mathcal{C}_{\epsilon}$ for $\frac{k}{n}>\alpha_{0}$.}
\end{minipage}
\end{figure}

We shall use a uniform version of the steepest descent method \cite{CFU}
as described in \cite[p.  369--376]{BlHa}, where the case of
two nearby saddle points is considered and the first step is to observe
that: 
\[
\widehat{b_{\lambda}^{n}}(k)\sim\frac{1}{2i\pi}\int_{\mathcal{C}_{\varepsilon}}\frac{e^{n\Phi(z)}}{z}\d z,\qquad n\rightarrow\infty,
\]
the contribution to the integral \eqref{eq:integral} from the part
of the contour outside $\mathcal{D}(-1,\varepsilon)$ being asymptotically
smaller than the integral itself \cite[ Subsection 5]{CFU}. This
can usually be proved by the familiar arguments of the ordinary method
of steepest descents, similar to those we previously used to prove
\eqref{eq:remainder_SDM}. Following \cite[(9.2.6)]{BlHa}, to simplify
the dependence of $z_{\pm}$ on $k/n$ we change the variable of integration
via a locally one-to-one transformation, implicitly given by $s=s_{a}(z)$
solving the equation 
\begin{equation}
\Phi(z)=-\left(\frac{1}{3}s^{3}-\gamma_{\alpha_{0}}^{2}s\right)+\eta,\label{eq:cubic_transfo}
\end{equation}
where the parameters $\gamma=\gamma_{\alpha_{0}}$ and $\eta$ are
determined in such a way that $s=0$ is mapped to $z=-1$ and the saddle points
$z_{\pm}$ are mapped symmetrically to $s=\pm\gamma$. For $z=z(s)$
to define a conformal map of $\mathcal{D}(-1,\varepsilon)$ it is
necessary that $\gamma_{\alpha_{0}}^{3}$ and $\eta$ be respectively
defined by \eqref{eq:def_g3_gen} and 
\[
\eta=\frac{\Phi(z_{+})+\Phi(z_{-})}{2},
\]
so that 
\[
\gamma^{2}=\gamma_{\alpha_{0}}^{2}=\frac{(1+\lambda)\left(\alpha_{0}-k/n\right)}{(\lambda(1-\lambda))^{1/3}}+o(\alpha_{0}-k/n),
\]
and 
\[
\eta=i\pi\left(1-\frac{k}{n}\right).
\]
For each value of $z$, \eqref{eq:cubic_transfo} defines three possible
values of $s,$ that is, there are three branches of the inverse transformation.
It is shown in \cite{CFU} that there is one branch of the transformation
\eqref{eq:cubic_transfo} that defines, for each $a$ in a neighborhood
of $\alpha_{0}$, a conformal map of $\mathcal{D}(-1,\varepsilon)$.
More precisely, the transformation \eqref{eq:cubic_transfo} has exactly
one branch $s=s(z,a)$ that can be expanded into a power series in
$z$ with coefficients that are continuous in $a$. On this branch
the points $z=z_{\pm}$ correspond to $s=\pm\gamma_{\alpha_{0}}$,
and the mapping of $z$ to $s$ is one-to-one on $\mathcal{D}(-1,\varepsilon)$.
This is an analog of \cite[Proposition 12]{SzZa2} and of \cite[Proposition 9]{SzZa1}
for $k/n$ in a neighborhood of $\alpha_{0}$ instead of $k/n$ close
to $\alpha_{0}^{-1}.$ Following \cite[Section 9.2]{BlHa} we get
\[
\frac{1}{2i\pi}\int_{{\cC_{\epsilon}}}\exp\left(n\Phi_{a}(z)\right)\frac{\d z}{z}=\frac{1}{2i\pi}\int_{{\hat{\cC}_{\epsilon}}}G_{0}(s)\exp\left(n\left(-\frac{s^{3}}{3}+\gamma^{2}s+\eta\right)\right)\d s
\]
where we made the notation less cluttered writing briefly $\gamma^{2}$
for $\gamma_{\alpha_{0}}^{2}$, and where 
\[
G_{0}(s)=\frac{1}{z(s)}\frac{\d z}{\d s}
\]
is regular on the image $\hat{D}(-1,\epsilon)$ of $D(-1,\epsilon)$
under the transformation $z\mapsto s(z)$. We exploit the fact that
if the integrand vanishes near a critical point then its contribution
to the asymptotic expansion is diminished. Thus we expand 
\[
G_{0}(s)=A_{0}+A_{1}s+(s^{2}-\gamma^{2})H_{0}(s),
\]
with $A_{0},\,A_{1},$ and $H_{0}$ to be determined. As long as $H_{0}$
is regular in $\hat{D}(-1,\epsilon)$ the last term of the above identity
vanishes at the two saddle points $s=\pm\gamma$. We can then determine
$A_{0},$ $A_{1}$ by setting $s=\pm\gamma$ in the above equality
to get 
\begin{equation}
A_{0}=\frac{G_{0}(\gamma)+G_{0}(-\gamma)}{2},\quad A_{1}=\frac{G_{0}(\gamma)-G_{0}(-\gamma)}{2\gamma}.\label{eq:A0A1}
\end{equation}
With $A_{0},$ $A_{1}$ defined by these formulas, it is shown in
\cite[p.~373]{BlHa} that $H_{0}=\frac{G_{0}(s)-A_{0}-A_{1}s}{s^{2}-\gamma^{2}}$
is regular in ${\hat{D}(-1,\epsilon)}$ as desired. We conclude that
$$
\int_{\cC_{\epsilon}}\exp\left(n\Phi_{a}(z)\right)\frac{\d z}{z}\sim e^{i\pi\left(n-k\right)}\int_{\hat{\cC}_{\epsilon}}(A_{0}+A_{1}s)\exp\left(n\left(-\frac{s^{3}}{3}+\gamma^{2}s\right)\right)\d s.\label{F}
$$
Following the procedure described in \cite[p.~371--375]{BlHa} we
consider a contour $C_{1}$ which is asymptotically equivalent to
$\hat{\mathcal{C}}_{\epsilon}$. This means that the contribution
of $C_{1}$ near the critical points coincides with that of $\hat{\mathcal{C}}_{\epsilon}$,
but $C_{1}$ continues to $\infty$ as a contour of steepest descent.
$C_{1}$ starts at infinity with points of argument $-2\pi/3$ and
ends at infinity with points of argument $2\pi/3$. See Figure~\ref{fig:asyEquiv},
Figure~\ref{fig:C_1_C_hat_1} and Figure~\ref{fig:C_1_C_hat_2}
below, for a description of $C_{1}$ and $\hat{\mathcal{C}}_{\epsilon}$.
We refer to \cite[Section 7.2]{BlHa} for a detailed description of
such contours.

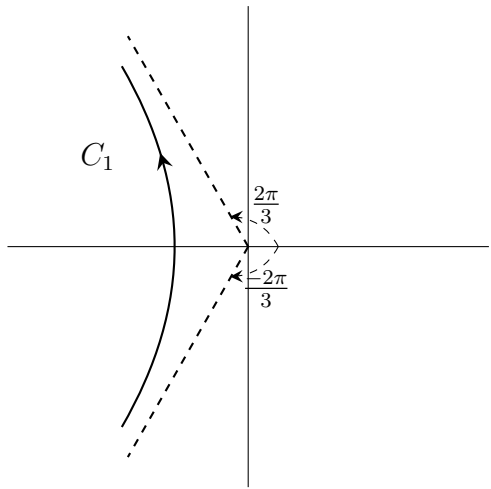
\begin{figure}[ht!]
\centering \begin{tikzpicture}[scale=0.8]
\tikzstyle arrowstyle=[scale=1.5]
\draw (-4,0) -- (4,0);
\draw (0,-4) -- (0,4);
\draw[thick,postaction={decorate,decoration={markings,mark=at position .75 with {\arrow[arrowstyle]{stealth}}}}] (-2.1,-3) to[bend right] (-2.1,3);
\draw[dashed,postaction={decorate,decoration={markings,mark=at position 1 with {\arrow[arrowstyle]{stealth}}}}] (0.5,0) to[bend right] (-0.3,0.5);
\draw[dashed,postaction={decorate,decoration={markings,mark=at position 1 with {\arrow[arrowstyle]{stealth}}}}] (0.5,0) to[bend left] (-0.3,-0.5);
\draw (0.3,0.7) node {$\frac{2 \pi}{3}$};
\draw (0.3,-0.7) node {$\frac{-2 \pi}{3}$};
\draw (-2.5,1.5) node {$C_1$};
\draw[thick, dashed] (0,0) to (-2,3.5);
\draw[thick, dashed] (0,0) to (-2,-3.5);
\end{tikzpicture} \caption{Introduction of the asymptotically equivalent contour $C_{1}$.}
\label{fig:asyEquiv} 
\end{figure}

\begin{figure}[H]
\centering %
\begin{minipage}[c]{0.5\textwidth}%
\begin{tikzpicture}[scale=0.8]
\tikzstyle arrowstyle=[scale=1.5]
\draw (-4,0) -- (4,0);
\draw (0,-4) -- (0,4);
\draw[thick,postaction={decorate,decoration={markings,mark=at position .75 with {\arrow[arrowstyle]{stealth}}}}] (-2.1,-3) to[bend right] (-2.1,3);
\draw[dashed,postaction={decorate,decoration={markings,mark=at position 1 with {\arrow[arrowstyle]{stealth}}}}] (0.5,0) to[bend right] (-0.3,0.5);
\draw[dashed,postaction={decorate,decoration={markings,mark=at position 1 with {\arrow[arrowstyle]{stealth}}}}] (0.5,0) to[bend left] (-0.3,-0.5);
\draw (0.3,0.7) node {$\frac{2 \pi}{3}$};
\draw (0.3,-0.7) node {$- \frac{2 \pi}{3}$};
\draw (-2.5,1.5) node {$C_1$};
\draw (1.24,-0.15) -- (1.24,0.15);
\draw (-1.24,-0.15) -- (-1.24,0.15);
\draw (1.1,-0.35) node [right] {$- \gamma$};
\draw (-1.15,-0.3) node [left] {$\gamma$};
\draw[thick, dashed] (0,0) to (-2,3.5);
\draw[thick, dashed] (0,0) to (-2,-3.5);
\draw[thick, blue] (-1.24,0) to (1.24,0);
\draw[thick, blue] (-0.8,0.6) node {$\hat{\mathcal{C_\varepsilon}}$};
\end{tikzpicture} \caption{\label{fig:C_1_C_hat_1} The contour $\hat{\mathcal{C}_{\epsilon}}$
for $\frac{k}{n}<\alpha_{0}$ and the asymptotically equivalent contour
$C_{1}$. }
\end{minipage}%
\begin{minipage}[c]{0.5\textwidth}%
\begin{tikzpicture}[scale=0.8]
\tikzstyle arrowstyle=[scale=1.5]
\draw (-4,0) -- (4,0);
\draw (0,-4) -- (0,4);
\draw[thick,postaction={decorate,decoration={markings,mark=at position .95 with {\arrow[arrowstyle]{stealth}}}}] (-0.88,-3) to[bend right] (-0.88,3);
\draw (-0.9,1.5) node {$C_1$};
\draw (-0.1,-0.4) -- (0.1,-0.4);
\draw (-0.1,0.4) -- (0.1,0.4);
\draw (0,-0.4) node [right] {$\gamma$};
\draw (0,0.4) node [right] {$-\gamma$};

\draw[thick, blue] (0,-0.5) to (0,0.5);
\draw[thick, blue] (-0.8,0.6) node {$\hat{\mathcal{C_\varepsilon}}$};
\end{tikzpicture} \caption{\label{fig:C_1_C_hat_2} The contour $\hat{\mathcal{C}_{\epsilon}}$
for $\frac{k}{n}>\alpha_{0}$ and the asymptotically equivalent contour
$C_{1}$. }
\end{minipage}
\end{figure}

When we replace $\hat{\cC}_{\epsilon}$ by $C_{1}$ in \eqref{F}, the introduced
error is negligible, since the integral of $(A_{0}+A_{1}t)\exp\left(n\left(-\frac{s^{3}}{3}+\gamma^{2}s\right)\right)$
over $C_{1}\setminus\hat{D}(1,\epsilon)$, is asymptotically smaller
than the integral over $\hat{\cC}_{\epsilon}$, see \cite[p. 372]{BlHa}
for details. The Airy function can be represented as an integral over
$C_{1}$. By a change of variable $\tau\mapsto i\tau$ and a deformation
of the contour of integration one obtains 
\begin{align*}
Ai(x) & =\frac{1}{2\pi}\int_{-\infty}^{+\infty}\cos\left(\frac{\tau^{3}}{3}+\tau x\right)\d\tau=\frac{1}{2i\pi}\int_{C_{1}}\exp\left(-\frac{u^{3}}{3}+ux\right)\d u
\end{align*}
and therefore

\begin{equation}
\frac{1}{2i\pi}\int_{{\cC_{\epsilon}}}\exp\left(n\Phi_{a}(z)\right)\frac{\d z}{z}\sim(-1)^{n-k}\left(\frac{A_{0}}{n^{1/3}}Ai(n^{2/3}\gamma^{2})+\frac{A_{1}}{n^{2/3}}Ai'(n^{2/3}\gamma^{2})\right),\qquad n\rightarrow\infty,\label{eq:Airy_expansion}
\end{equation}
where $A_{0},\,A_{1}$ are defined in \eqref{eq:A0A1}. To compute
$A_{0},\,A_{1}$ we write 
\[
G_{0}(\pm\gamma)=G_{0}(s_{\pm})=\frac{1}{z_{\pm}}z'(s_{\pm}).
\]
A computation (see the proof of \cite[Proposition 17]{SzZa4} for
more details) shows that: 
\[
z'(t_{\pm})=z_{\pm}\sqrt{\frac{2\abs{\gamma}}{a}}\frac{1}{\abs{\Delta}^{1/4}},\quad\textnormal{where}\ a=k/n,\quad\textnormal{and}\ \Delta=\left(a-\alpha_{0}\right)\left(\alpha_{0}^{-1}-a\right).
\]
Therefore 
\[
G_{0}(\gamma)=G_{0}(-\gamma)=\sqrt{\frac{2\abs{\gamma}}{a}}\frac{1}{\abs{\Delta}^{1/4}}
\]
and 
\[
A_{0}=\frac{G_{0}(\gamma)+G_{0}(-\gamma)}{2}=G_{0}(\gamma)=\sqrt{\frac{2\abs{\gamma}}{a}}\frac{1}{\abs{\Delta}^{1/4}},
\]
\[
A_{1}=\frac{G_{0}(\gamma)-G_{0}(-\gamma)}{2\gamma}=0.
\]
\end{proof}

\subsection{Proofs of Theorem \ref{Th_Regions_I_VII}~(3),~(4) and of
Theorem \ref{Th:Regions_IV_V_VI}(2)}

\subsubsection{The case where $a=k/n$ is close to the boundaries $\alpha_{0},\alpha_{0}^{-1}$.}

We first discuss the situation where $a=k/n$ approaches the boundaries
$\alpha_{0},\alpha_{0}^{-1}$ and start by proving Theorem \ref{Th_Regions_I_VII}~(3),~(4). Here we apply Proposition
\ref{Prop_Airy_fcts} together with \eqref{eq:exp_beh_Airy}. 

\begin{proof}[Proof of Theorem \ref{Th_Regions_I_VII}~{\rm (3),~(4)}.]
First we prove part (3). If $k\in[\alpha n,\alpha_{0}n-n^{1/3})$
and if, in addition, $n^{2/3}(\alpha_{0}-k/n)\rightarrow+\infty$
(Region III) then $n^{2/3}\gamma^{2}\rightarrow+\infty$ as $n$ tends
to $\infty$. Since $Ai(x)\sim\frac{1}{2x^{1/4}\sqrt{\pi}}\exp\left(-\frac{2}{3}x^{3/2}\right)$
as $x\rightarrow+\infty$, we have 
\begin{align*}
\sqrt{\frac{2\abs{\gamma}}{k/n}}\frac{(-1)^{n-k}}{\abs{\Delta}^{1/4}}Ai(n^{2/3}\gamma^{2}) & \sim\sqrt{\frac{2\abs{\gamma}}{k/n}}\frac{(-1)^{n-k}}{\abs{\Delta}^{1/4}}\frac{1}{2\sqrt{\pi}n^{1/6}\abs{\gamma}^{1/2}}\exp\left(-\frac{2}{3}n\abs{\gamma}^{3}\right)\\
 & \sim\frac{1}{\sqrt{2\pi}}\frac{(-1)^{n-k}}{\sqrt{k/n}\left[(\alpha_{0}-k/n)(\alpha_{0}^{-1}-k/n)\right]^{1/4}}\frac{\exp\left(-\frac{2}{3}n\abs{\gamma}^{3}\right)}{n^{1/6}}.
\end{align*}
It remains to use \eqref{eq:Airy_expansion} and to divide both parts 
by $n^{1/3}.$ We omit the proof of part (4) which is almost identical. 
\end{proof}
Next we apply Proposition \ref{Prop_Airy_fcts} together with \eqref{eq:osc_beh_Airy}
to prove Theorem \ref{Th:Regions_IV_V_VI}~(2) for $k\leq\beta n$
or $k\geq\beta^{-1}n$. 

\begin{proof}[Proof of Theorem \ref{Th:Regions_IV_V_VI}~{\rm (2)} for $k\leq\beta n$
or $k\geq\beta^{-1}n$.]
Let $k\in(\alpha_{0}n+n^{1/3},\beta n]\cup[\beta^{-1}n,\alpha_{0}^{-1}n-n^{1/3})$.
We assume in addition that either $n^{2/3}(k/n-\alpha_{0})\rightarrow+\infty$
or $n^{2/3}(\alpha_{0}^{-1}-k/n)\rightarrow+\infty$ (i.e. $k$ lies
in Region IV\textbackslash$(\beta n,\beta^{-1}n)${} ):

i) If $n^{2/3}(k/n-\alpha_{0})\rightarrow+\infty$ then $\gamma^{2}=\gamma_{\alpha_{0}}^{2}$
and $n^{2/3}\gamma^{2}\rightarrow-\infty$. Recalling that $Ai(-x)\sim\frac{1}{x^{1/4}\sqrt{\pi}}\cos\left(\frac{2}{3}x^{3/2}-\frac{\pi}{4}\right)$
as $x\rightarrow+\infty$ we obtain 
\[
\sqrt{\frac{2\abs{\gamma}}{k/n}}\frac{(-1)^{n-k}}{\abs{\Delta}^{1/4}}Ai(n^{2/3}\gamma^{2})\sim\sqrt{\frac{2\abs{\gamma}}{k/n}}\frac{(-1)^{n-k}}{\abs{\Delta}^{1/4}}\frac{1}{\sqrt{\pi}n^{1/6}\abs{\gamma}^{1/2}}\cos\left(\frac{2}{3}n\abs{\gamma}^{3}-\frac{\pi}{4}\right),
\]
and therefore 
\begin{multline*}
\sqrt{\frac{2\abs{\gamma}}{k/n}}\frac{(-1)^{n-k}}{\abs{\Delta}^{1/4}}Ai(n^{2/3}\gamma^{2})\\
\sim\sqrt{\frac{2}{\pi}}\frac{(-1)^{n-k}}{\sqrt{k/n}\left[(\alpha_{0}-k/n)(\alpha_{0}^{-1}-k/n)\right]^{1/4}}\frac{\cos\left(n\abs{ih(\varphi_{+})-i\pi(1-k/n)}-\frac{\pi}{4}\right)}{n^{1/6}},
\end{multline*}
where we use the definitions of $\gamma^{3}$ (see \eqref{eq:exact_g_3_left_bd})
and $h$ (see \eqref{eq:h}). Using the fact that $\gamma^{3}\in i\mathbb{R}_{+}$
we obtain 
\begin{align*}
\cos\left(n\abs{ih(\varphi_{+})-i\pi(1-k/n)}-\frac{\pi}{4}\right) & =\cos\left(nh(\varphi_{+})-\pi(n-k)-\frac{\pi}{4}\right)\\
 & =(-1)^{n-k}\cos\left(nh(\varphi_{+})-\frac{\pi}{4}\right).
\end{align*}
It remains to use Proposition \ref{Prop_Airy_fcts} and to divide
by $n^{1/3}$.

ii) If $n^{2/3}(\alpha_{0}^{-1}-k/n)\rightarrow+\infty$, then our
argument is similar. We use the second formula in Proposition \ref{Prop_Airy_fcts}
and the fact that this time $\gamma^{3}=\gamma_{\alpha_{0}^{-1}}^{3}=-\frac{3}{2}ih(\varphi_{+})$,
see \eqref{eq:exact_g_3_right_bd}. 
\end{proof}

\subsubsection{\label{subsec:Fedoryuk}The case where $a=k/n$ is separate from
the boundaries $\alpha_{0},\alpha_{0}^{-1}$.}

Lemma \ref{lem:critical_pts_f} shows that the location of stationary
points of $\Phi_{a}$ in $\mathbb{C}$ is determined by the location
of $a$ relative to the critical interval $[\alpha_{0},\alpha_{0}^{-1}]$.
The situation where $a$ approaches the boundaries $\alpha_{0},\alpha_{0}^{-1}$
was discussed in the previous subsection. In this case, the stationary
points $z_{\pm}$ degenerate and uniform methods are required. The
situation where $a$ is separate from $\alpha_{0},\alpha_{0}^{-1}$,
that is there exists $\beta\in(\alpha_{0},1)$ that separates $a$
from the boundary, $a\in[\beta,\beta^{-1}]$, is different and even
simpler. In this case the stationary points $z_{\pm}=e^{i\varphi_{\pm}}$
of $\Phi_{a}$ belong to the contour of integration $\partial\mathbb{D}$
and remain separate from $\pm1$, see below. Since $\Abs{z^{-k/n}\frac{z-\lambda}{1-\lambda z}}=1$
for any $z\in\partial\mathbb{D}$ we can introduce the real function
\[
\tilde{h}(\varphi)=\widetilde{h_{a}}(\varphi)=-h_{a}(\varphi)=i\Phi_{a}(e^{i\varphi}),\qquad\varphi\in[0,\pi],
\]
to write the integral as a generalized Fourier integral (the Fourier/Taylor
coefficients of $b_{\lambda}^{n}$ are real because $\lambda\in(0,1)$),
\begin{align*}
\overline{\widehat{b_{\lambda}^{n}}(k)}={\frac{1}{2\pi}\int_{-\pi}^{\pi}e^{-n\Phi_{a}(e^{i\varphi})}\d\varphi=\frac{1}{\pi}\Re{\left\{ \int_{0}^{\pi}e^{in\widetilde{h_{a}}(\varphi)}\d\varphi\right\} .}}
\end{align*}
The asymptotic behavior of this integral can be determined using A.
Erdélyi's standard \textit{method of stationary phase}~\cite{Erd}
and the approach from \cite[Section 3.1]{SzZa3}, which will be done
at the end of this section. Before that let us mention that a more
elaborate version of the classical method of stationary phase, due
to M.V.~Fedoryuk \cite[Theorem 2.4 p.~80]{Fed2} (see \cite[Theorem 1.6 p.107]{Fed1}
for a simple version in one dimension), will make our proof much shorter.
Moreover, Fedoryuk's method immediately provides us with a sharp error
term for the first order approximation of $\widehat{b_{\lambda}^{n}}(k)$,
which holds uniformly for $k\in[\beta n,\beta^{-1}n]$. We first provide
this simple proof making use of Fedoryuk's result, and then write
in full detail a classical (but longer and more technical) proof of
the same formula, using A.Erdélyi's standard method of stationary
phase. 
\begin{proof}[Proof of Theorem \ref{Th:Regions_IV_V_VI}~{\rm (2)} for $\beta n\leq k\leq\beta^{-1}n$
using Fedoryuk's method. ]
Suppose that $a=k/n\in[\beta,\beta^{-1}]$. The stationary points
of $\tilde{h}=\widetilde{h_{a}}$ are given by 
\[
z_{\pm}=\frac{a(1+\lambda^{2})-(1-\lambda^{2})}{2\lambda a}\pm i\sqrt{1-\left(\frac{a(1+\lambda^{2})-(1-\lambda^{2})}{2\lambda a}\right)^{2}}\in\partial\mathbb{D}
\]
and we write $z_{\pm}=e^{i\varphi_{\pm}}$ with $\varphi_{+}\in[0,\pi]$
and $\varphi_{-}\in(-\pi,0]$. Only $z_{+}$ is relevant since we
integrate over $[0,\pi]$ and the unique critical point $\varphi_{+}=\varphi_{+}(k/n)$
of $\widetilde{h_{a}}$ in $(0,\pi)$ satisfies $x\leq\varphi_{+}\leq\pi-x$
for some $x=x(\beta,\lambda)>0$ because 
\[
\Abs{e^{i\varphi_{+}}-1}\geq(1-\lambda)\sqrt{\frac{\beta}{\lambda}}\sqrt{\alpha_{0}^{-1}-\beta^{-1}},\qquad\Abs{e^{i\varphi_{+}}+1}\geq(1+\lambda)\sqrt{\frac{\beta}{\lambda}}\sqrt{\beta-\alpha_{0}}.
\]
These inequalities follow from the identities 
\[
\Abs{e^{i\varphi_{+}}-1}^{2}=\frac{(1-\lambda)^{2}\left(\alpha_{0}^{-1}-a\right)}{a\lambda},\qquad\Abs{e^{i\varphi_{+}}+1}^{2}=\frac{(1+\lambda)^{2}\left(a-\alpha_{0}\right)}{a\lambda}.
\]
For the second derivative we have 
\[
-i\tilde{h}''(\varphi)=\frac{\partial}{\partial\varphi}\left(\frac{\partial\Phi}{\partial z}\frac{dz}{d\varphi}\right)=\frac{\partial^{2}\Phi}{\partial z^{2}}\left(\frac{dz}{d\varphi}\right)^{2}+\frac{\partial\Phi}{\partial z}\frac{d^{2}z}{(d\varphi)^{2}}.
\]
It follows from \eqref{eq:sec_der_Phi_z_p} that 
\begin{eqnarray*}
i\frac{\partial^{2}\tilde{h}}{\partial\varphi^{2}}\Bigg|_{\varphi=\varphi_{+}} & = & z_{+}^{2}\Phi''(z_{+}),
\end{eqnarray*}
which gives, by \eqref{eq:sec_der_Phi_z_p_2}, that 
\[
\tilde{h}''(\varphi_{+})=\frac{k}{n}\sqrt{\left(\frac{k}{n}-\alpha_{0}\right)\left(\alpha_{0}^{-1}-\frac{k}{n}\right)}\geq\min_{a\in[\beta,\beta^{-1}]}a\sqrt{(a-\alpha_{0})(\alpha_{0}^{-1}-a)}=:C(\beta,\lambda)>0.
\]
We are now ready to apply Fedoryuk's result \cite[Theorem 2.4 p.~80]{Fed2}
with $d=1$ and $\Omega=(0,\pi)$ to 
\[
I(n,a)=\int_{0}^{\pi}\nu(\varphi)e^{in\tilde{h}(\varphi)}\d\varphi,
\]
where $\tilde{h}=\widetilde{h_{a}},$ $a=k/n\in[\beta,\beta^{-1}]=:M$,
$x=x(\beta,\lambda)>0$ and $\nu:\left[0,\pi\right]\longrightarrow\mathbb{R}$
is a \textit{neutralizer} satisfying ${\nu=1}$ on $\left[x/2,\pi-x/2\right],$
${\nu=0}$ on $\left[0,x/4\right]\cup\left[\pi-x/4,\pi\right]$ and
$0\leqslant\nu\leqslant1$. The compact $K:=\left[\frac{x}{4},\pi-\frac{x}{4}\right]$
satisfies Assumption 2 in \cite{Fed2}. The function $\nu:\varphi\mapsto\nu(\varphi)$
does not depend neither on $a$ nor on $\xi=n$ and Assumption 3 in
\cite{Fed2} is satisfied with $m=0.$ Finally for $a=\frac{k}{n}\in M$
the unique critical point $\varphi_{+}=\varphi_{+}(a)$ of $\varphi\mapsto\widetilde{h_{a}}(\varphi)$
satisfies 
\[
\tilde{h}''(\varphi_{+})\geq C(\beta,\lambda)>0
\]
and Assumptions 4 and 5 in \cite{Fed2} are also satisfied. Applying
Fedoryuk's asymptotic formula with $l=1$, $\alpha_{1}=\frac{3}{2}$,
$b_{1}=\sqrt{2\pi}\left(\tilde{h}''(\varphi_{+})\right){}^{-1/2}\exp\left(\frac{i\pi}{4}\right)$,
we obtain that 
\[
I(n,a)=\sqrt{2\pi}\left(\tilde{h}''(\varphi_{+})\right){}^{-1/2}\exp\left(\frac{i\pi}{4}\right)n^{-1/2}e^{in\tilde{h}(\varphi_{+})}+\cO\left(n^{-3/2}\right)
\]
where $\cO\left({n^{-3/2}}\right)$ is uniform over $k/n\in[\beta,\beta^{-1}]$.
It remains to observe that $\int_{0}^{\pi}e^{in\widetilde{h_{a}}(\varphi)}\d\varphi-I(n,a)=\cO(n^{-2})$
uniformly for $k/n\in[\beta,\beta^{-1}]$ to conclude that 
\begin{equation}
\widehat{b_{\lambda}^{n}}(k)=\sqrt{\frac{2}{n\pi}}\frac{\cos\left(n\tilde{h}(\varphi_{+})+\pi/4\right)}{\sqrt{k/n}\left[(\alpha_{0}^{-1}-k/n)(k/n-\alpha_{0})\right]^{1/4}}+\cO\left(n^{-3/2}\right)\label{eq:asym_exp_centr_reg},
\end{equation}
where $\cO\left({n^{-3/2}}\right)$ is uniform over $k/n\in[\beta,\beta^{-1}]$. 
\end{proof}
For the sake of completeness we end this section by proving the above
asymptotic expansion \eqref{eq:asym_exp_centr_reg} using the standard
\textit{method of stationary phase}~\cite{Erd}. 
\begin{proof}[Proof of Theorem \ref{Th:Regions_IV_V_VI}~{\rm (2)} for $\beta n\leq k\leq\beta^{-1}n$
using Erdélyi's method.]

To determine the asymptotic behavior we apply a standard result of
A.~Erdélyi~\cite[Theorem 4]{Erd} (see also \cite[Theorem 1.3]{AlDe1}
for a detailed discussion of this result and the involved error estimates),
which however requires that the stationary point is an endpoint of
the interval of integration. Hence we begin by splitting our generalized
Fourier integral: 
\[
\int_{0}^{\pi}e^{in\tilde{h}(\varphi)}\d\varphi=\int_{0}^{\varphi_{+}}e^{in\tilde{h}(\varphi)}\d\varphi+\int_{\varphi_{+}}^{\pi}e^{in\tilde{h}(\varphi)}\d\varphi.
\]
For the second integral, Theorem 4 of \cite{Erd} yields 
\begin{align*}
\int_{\varphi_{+}}^{\pi}e^{in\tilde{h}(\varphi)}\d\varphi & =\frac{1}{2}\Gamma(1/2)\kappa_{1}(0)e^{i\frac{\pi}{4}}n^{-1/2}e^{in\tilde{h}(\varphi_{+})}+\frac{1}{2}\Gamma(1)\kappa_{1}'(0)e^{i\frac{\pi}{2}}n^{-1}e^{in\tilde{h}(\varphi_{+})}\\
 & +\frac{1}{2}\Gamma(3/2)\kappa_{1}''(0)e^{i\frac{3\pi}{4}}n^{-3/2}e^{in\tilde{h}(\varphi_{+})}+e^{inh(\pi)}\frac{i}{n}\frac{1}{\tilde{h}'(\pi)}\\
 & +R_{3}^{(1)}(n)+R_{3}^{(2)}(n),
\end{align*}
where 
\[
\kappa_{1}(0)=2^{1/2}\left(\tilde{h}''(\varphi_{+})\right)^{-1/2},
\]
\[
\kappa_{1}'(0)=-\frac{2}{h''(\varphi_{+})}\frac{\tilde{h}^{(3)}(\varphi_{+})}{3\tilde{h}''(\varphi_{+})},
\]
\[
\kappa_{1}''(0)=\frac{2^{5/2}}{\left(\tilde{h}''(\varphi_{+})\right)^{3/2}}\left(\frac{5}{36}\left(\tilde{h}^{(3)}(\varphi_{+})\right)^{2}-\frac{\tilde{h}''(\varphi_{+})\tilde{h}^{(4)}(\varphi_{+})}{12}\right)\frac{3}{4\tilde{h}''(\varphi_{+})^{2}},
\]
the error terms $R_{3}^{(1)}(n),\,R_{3}^{(2)}(n)$ will be explicitly
estimated from above in what follows, according to \cite[Theorem 1.3]{AlDe1},
and the function $\kappa_{1}$ will be explicitly defined later on.
First of all we observe that for $k\in\left[\beta n,\,\beta^{-1}n\right]$
we have 
\[
\tilde{h}''(\varphi_{+})\geq C(\beta,\lambda)>0
\]
which shows in particular that 
\[
\frac{1}{2}\Gamma(3/2)\kappa_{1}''(0)e^{i\frac{3\pi}{4}}n^{-3/2}e^{in\tilde{h}(\varphi_{+})}=\cO\left({n^{-3/2}}\right)
\]
and that the $\cO\left({n^{-3/2}}\right)$-term is uniform over $a=\frac{k}{n}\in[\beta,\beta^{-1}]$.
Observing that $\tilde{h}'(\pi)=\frac{k}{n}-\alpha_{0}$ we find that
the third term in the expansion of $\int_{\varphi_{+}}^{\pi}e^{in\tilde{h}(\varphi)}\d\varphi$
is purely imaginary. The second one will cancel out when we will add
to $\int_{\varphi_{+}}^{\pi}e^{in\tilde{h}(\varphi)}\d\varphi$, the
integral $\int_{0}^{\varphi_{+}}e^{in\tilde{h}(\varphi)}\d\varphi$
whose asymptotic expansion is computed below, see \eqref{eq:exp_1st_int_erd}.
Now, we show that the error terms $R_{3}^{(1)}(n),\,R_{3}^{(2)}(n)$
both satisfy 
\[
R_{3}^{(j)}(n)=\cO\left({n^{-3/2}}\right),\qquad j=1,2
\]
again uniformly for $a=\frac{k}{n}\in[\beta,\beta^{-1}]$. To this
aim we first recall that for $k\in\left[\beta n,\,\beta^{-1}n\right]$
the unique critical point $\varphi_{+}$ of $\tilde{h}$ satisfies
$x\leq\varphi_{+}\leq\pi-x$ for some $x=x(\beta,\lambda)>0.$ We
use the notation from \cite[Section 1]{AlDe1} and choose $\eta=\frac{x}{4}\in\left(0,\frac{\pi-\varphi_{+}}{2}\right).$
For $j=1,2,$ let $\psi_{j}=I_{j}\to\mathbb{R}$ be the
functions defined by 
\[
\psi_{1}(\varphi)=\left(\tilde{h}(\varphi)-\tilde{h}\left(\varphi_{+}\right)\right)^{\frac{1}{2}},\quad\psi_{2}(\varphi)=\tilde{h}\left(\pi\right)-\tilde{h}(\varphi)
\]
$\text{with }I_{1}:=\left[\varphi_{+},\pi-\eta\right],I_{2}=\left[\varphi_{+}+\eta,\pi\right]\text{ and }s_{1}:=\psi_{1}\left(\pi-\eta\right),s_{2}:=\psi_{2}\left(\varphi_{+}+\eta\right)$.
$\psi_{j}$ is shown to be a diffeomorphism between $I_{j}$ and $[0,s_{j}]$,
see \cite[Proposition 3.2]{AlDe1}. For $j=1,2,$ let $\kappa_{j}:\left(0,s_{j}\right]\rightarrow\mathbb{C}$
be the functions defined by

\[
\kappa_{j}(s):=\left(\psi_{j}^{-1}\right)^{\prime}(s).
\]
It is shown in \cite[Proposition 3.3]{AlDe1} that $\kappa_{j}$ can
be continuously extended to $[0,s_{j}]$ and that $\kappa_{j}\in\mathcal{C}^{3}([0,s_{j}])$.
Let $\nu:\left[\varphi_{+},\pi\right]\longrightarrow\mathbb{R}$ be
a neutralizer such that ${\nu=1}$ on $\left[\varphi_{+},\varphi_{+}+\eta\right]$,
${\nu=0}$ on $\left[\pi-\eta,\pi\right]$ and $0\leqslant\nu\leqslant1$,
where $\eta$ is defined above. For $j=1,2,$ let $\nu_{j}=\left[0,s_{j}\right]\rightarrow\mathbb{R}$
be the functions defined by

\[
\nu_{1}(s)=\nu\circ\psi_{1}^{-1}(s)\quad,\quad v_{2}(s)=(1-\nu)\circ\psi_{2}^{-1}(s).
\]
It is shown in \cite[Theorem 1.3]{AlDe1} that 
\[
\left|R_{3}^{(j)}(n)\right|\leqslant\frac{1}{4}\Gamma\left(\frac{3}{2}\right)n^{-\frac{3}{2}}\int_{0}^{s_{j}}\left|\frac{d^{3}}{ds^{3}}\left[\nu_{j}\kappa_{j}\right](s)\right|{\rm d}s,\qquad j=1,2.
\]
To prove that $R_{3}^{(j)}(n)=\cO\left({n^{-3/2}}\right)$ uniformly
for $\frac{k}{n}\in[\beta,\beta^{-1}]$ we write 
\[
\int_{0}^{s_{j}}\left|\frac{d^{3}}{ds^{3}}\left[\nu_{j}\kappa_{j}\right](s)\right|{\rm d}s\leq s_{j}\max_{s\in[0,s_{j}]}\left|\frac{d^{3}}{ds^{3}}\left[\nu_{j}\kappa_{j}\right](s)\right|,\qquad j=1,2.
\]

\textit{First, we treat in details the case $j=2$.} We need to show 
that 
\begin{equation}
\max_{s\in[0,s_{2}]}\left|\kappa_{2}^{(l)}(s)\right|=\cO\left(1\right),\qquad l=0,\dots,3,\label{eq:AS2}
\end{equation} 
uniformly for $\frac{k}{n}\in[\beta,\beta^{-1}]$. We have 
\[
\kappa_{2}(s)=-\frac{1}{\tilde{h}'(\psi_{2}^{-1}(s))}.
\]
Computing the derivatives of $\kappa_{2}$ and taking into account
that $\tilde{h}$ and each of its derivatives are uniformly bounded
on $[0,\,\pi]\supset[\varphi_{+}+\eta,\,\pi]$, we observe that the
proof of \eqref{eq:AS2} follows from the fact that $\min_{s\in[0,s_{2}]}\abs{\tilde{h}'(\psi_{2}^{-1}(s))}$
is uniformly separated from $0$. More precisely, for $s\in[0,s_{2}]$,
$\psi_{2}^{-1}(s)\in[\varphi_{+}+\eta,\,\pi]$ and for $\varphi\in(0,\pi)$,
we have 
\[
\tilde{h}''(\varphi)=\frac{2\lambda(1-\lambda^{2})\sin(\varphi)}{(1+\lambda^{2}-2\lambda\cos(\varphi))^{2}}>0
\]
which implies that $\tilde{h}'$ is increasing on $[\varphi_{+}+\eta,\,\pi]$.
We know that $h'(\varphi_{+})=0$, $h'$ does not vanish on $(\varphi_{+},\pi]$
and $\tilde{h}'(\pi)=\frac{k}{n}-\alpha_{0}>0$. Therefore, $\min_{s\in[0,s_{2}]}\abs{\tilde{h}'(\psi_{2}^{-1}(s))}=\tilde{h}'(\varphi_{+}+\eta)$.
By the mean-value theorem there is $\theta\in(\varphi_{+},\,\varphi_{+}+\eta)$
such that 
\begin{align*}
\tilde{h}'(\varphi_{+}+\eta) & =\tilde{h}'(\varphi_{+}+\eta)-\tilde{h}'(\varphi_{+})\\
 & =\eta\tilde{h}''(\theta)\geq\eta\frac{2\lambda(1-\lambda^{2})\sin\theta}{(1+\lambda^{2}-2\lambda\cos\theta)^{2}}\\
 & \ge\frac{x\lambda(1+\lambda)}{2(1-\lambda)^{3}}\min_{t\in[x,\pi-3x/4]}\sin(t)
\end{align*}
because $x\leq\varphi_{+}\leq\theta\leq\varphi_{+}+\eta\leq\pi-\frac{3x}{4}$.
The same type of argument yields 
\[
\max_{s\in[0,s_{2}]}\left|\nu_{2}^{(l)}(s)\right|=\cO\left(1\right),\qquad0\leq l\leq3,
\]
uniformly for $\frac{k}{n}\in[\beta,\beta^{-1}]$. Indeed, a direct
computation shows that $\nu_{2}'(s)=-\frac{(1-\nu)'(\psi_{2}^{-1}(s))}{\tilde{h}'(\psi_{2}^{-1}(s))}$
and $\nu_{2}'$ is of the same nature as $\kappa_{2}$. We conclude
that $R_{3}^{(2)}(n)=\cO\left({n^{-3/2}}\right)$ uniformly for $\frac{k}{n}\in[\beta,\beta^{-1}]$.

\textit{Now, we deal with the case $j=1.$ }We apply the same type
of reasoning to show that $R_{3}^{(1)}(n)=\cO\left({n^{-3/2}}\right)$
uniformly for $\frac{k}{n}\in[\beta,\beta^{-1}]$. First of all, since
$\nu_{1}(s)=1$ in some neighborhood of $s=0$ and $\nu_{1}'(s)=-\frac{\nu'(\psi_{1}^{-1}(s))}{\tilde{h}'(\psi_{1}^{-1}(s))}$,
we have 
\[
\max_{s\in[0,s_{1}]}\left|\nu_{1}^{(l)}(s)\right|=\cO\left(1\right),\qquad 0\le l\le 3,
\]
uniformly for $\frac{k}{n}\in[\beta,\beta^{-1}]$. Indeed, if $s$
is separated from 0, a direct computation shows that $\nu_{1}^{(l)}$
are expressed as quotients whose numerators are uniformly bounded
from above and whose denominators are powers of $\tilde{h}'(\psi_{1}^{-1})$
which are therefore uniformly separated from 0 (this can be seen, for
example, by an application of the mean-value theorem as above). For
$s\in[0,s_{1}]$, we have $\psi_{1}^{-1}(s)\in[\varphi_{+},\pi-\eta]$
and 
\[
\kappa_{1}(s)=\frac{2s}{\tilde{h}'(\psi_{1}^{-1}(s))}.
\]
To show that
\begin{equation}
\max_{s\in[0,s_{1}]}\left|\kappa_{1}^{(l)}(s)\right|=\cO\left(1\right),\qquad0\leq l\leq3,\label{eq:AS1}
\end{equation}
uniformly for $\frac{k}{n}\in[\beta,\beta^{-1}]$, we begin with a
series of preliminary observations. First, for $0\leq l\leq3$, the
functions $\kappa_{1}^{(l)}$ are continuous on the compact
$[0,s_{1}]$, see \cite[Proposition 3.3]{AlDe1}. Therefore, the function
$s\mapsto\abs{\kappa_{1}^{(l)}(s)}$ attains its maximum on this interval.
Second, we recall that the three explicit formulas we have previously
written for $\kappa_{1}^{(l)}(0)$, $0\leq l\leq2,$ show that these
quantities are expressed as quotients whose numerators are uniformly
bounded (because $\tilde{h}$ and its derivatives are bounded) and
whose denominators are expressed as powers of $\tilde{h}'(\varphi_{+})\geq C(\beta,\lambda)>0$.
Therefore for $0\leq l\leq3$, we have $\abs{\kappa_{1}^{(l)}(0)}=\cO\left({1}\right)$
uniformly for $\frac{k}{n}\in[\beta,\beta^{-1}].$ Third, if $s$
is separated from 0, a direct computation shows again that $\kappa_{1}^{(l)}$
are expressed as quotients whose numerators are uniformly bounded
from above and whose denominators are powers of $h'(\psi_{1}^{-1})$
which are therefore uniformly separated from 0. We use these observations
to prove that for any $0\leq l\leq3,$ \eqref{eq:AS1} holds uniformly
for $\frac{k}{n}\in[\beta,\beta^{-1}]$. We only provide a proof of
\eqref{eq:AS1} for the case $l=0$, the other cases $1\leq l\leq3,$
can be proved similarly. Let $t=t(n)\in[0,s_{1}]$ be such that 
\[
\max_{s\in[0,s_{1}]}\left|\kappa_{1}(s)\right|=\abs{\kappa_{1}(t(n))}=\abs{\kappa_{1}(\psi_{1}(\varphi(n)))},
\]
where $\varphi(n)\in[x,\pi-\eta]$. If $\abs{\kappa_{1}(t(n))}$ is
not uniformly bounded for $\frac{k}{n}=\frac{k(n)}{n}\in[\beta,\beta^{-1}]$
as $n$ tends to $\infty$, then $\abs{\kappa_{1}(t(n_{l}))}\rightarrow_{l}\infty$
for some subsequence $(n_{l})_{l}$ and $\frac{k(n_{l})}{n_{l}}\in[\beta,\beta^{-1}].$
A direct computation shows that 
\[
\tilde{h}'(\psi_{1}^{-1}(t(n_{l})))=\frac{k(n_{l})}{n_{l}}-\frac{1-\lambda^{2}}{1+\lambda^{2}-2\lambda\cos\left(\varphi(n_{l})\right)}.
\]
By compactness, we can construct a new subsequence $(n_{q})$ (actually
extracted from $(n_{l})$) such that both $\frac{k(n_{q})}{n_{q}}$
converges to some $\tilde{\beta}\in[\beta,\beta^{-1}]$ and $\varphi(n_{q})$
converges to some $\tilde{\varphi}\in[x,\pi-\eta]$. Passing to the
limit as $q$ tends to $\infty$ we find that 
\[
\lim_{q}\tilde{h}'(\psi_{1}^{-1}(t(n_{q})))=\tilde{h}_{\tilde{\beta}}'(\tilde{\varphi})=\tilde{\beta}-\frac{1-\lambda^{2}}{1+\lambda^{2}-2\lambda\cos\left(\tilde{\varphi}\right)}.
\]
Therefore, 
\[
\lim_{q}\kappa_{1}(t(n_{q}))=\tilde{\kappa_{1}}(\tilde{\psi_{1}}(\tilde{\varphi})),
\]
where $\tilde{\psi_{1}}(\varphi)=\sqrt{\tilde{h}_{\tilde{\beta}}(\varphi)-\tilde{h}_{\tilde{\beta}}(\varphi_{+})}$, 
$\tilde{h}'(\varphi_{+})=0$ and $\tilde{\kappa_{1}}(s)=\frac{2s}{\tilde{h}_{\tilde{\beta}}'(\tilde{\psi_{1}}^{-1}(s))}$. 
This contradicts the assumption $\lim_{q\rightarrow\infty}\abs{\kappa_{1}(t(n_{q}))}=\infty$.

The analysis of the first integral $\int_{0}^{\varphi_{+}}e^{in\tilde{h}(\varphi)}\d\varphi$
is essentially the same but we change the variable of integration
$\varphi\mapsto-\varphi$ as suggested in \cite[p.~23]{Erd}. We get
\[
\int_{0}^{\varphi_{+}}e^{in\tilde{h}(\varphi)}\d\varphi=\int_{-\varphi_{+}}^{0}e^{in\tilde{h}(-\varphi)}\d\varphi.
\]
Applying Theorem 4 of \cite{Erd} (together with \cite[Theorem 1.3]{AlDe1}
to estimate the $\cO-$term), we obtain that 
\begin{align}
\int_{-\varphi_{+}}^{0}e^{in\tilde{h}(-\varphi)}\d\varphi & =\frac{1}{2}\Gamma(1/2)\kappa_{3}(0)e^{i\frac{\pi}{4}}n^{-1/2}e^{in\tilde{h}(\varphi_{+})}+\frac{1}{2}\Gamma(1)\kappa_{3}(0)e^{i\frac{\pi}{2}}n^{-1}e^{in\tilde{h}(\varphi_{+})}\label{eq:exp_1st_int_erd}\\
 & -\frac{i}{n}e^{in\tilde{h}(0)}\frac{1}{\tilde{h}'(0)}+\cO\left({n^{-3/2}}\right)\nonumber 
\end{align}
with 
\[
\kappa_{3}(0)=2^{1/2}\left(\tilde{h}''(\varphi_{+})\right)^{-1/2},
\]
\[
\kappa_{3}'(0)=\frac{2}{\tilde{h}''(\varphi_{+})}\frac{\tilde{h}^{(3)}(\varphi_{+})}{3\tilde{h}''(\varphi_{+})},
\]
where, as for the above asymptotic expansion of $\int_{\varphi_{+}}^{\pi}e^{in\tilde{h}(\varphi)}\d\varphi$,
the $\cO\left({n^{-3/2}}\right)$-term is again uniform over $a=\frac{k}{n}\in[\beta,\beta^{-1}]$.
Observing that $\tilde{h}(0)=0$, $\tilde{h}(\pi)=(a-1)\pi,$ $\tilde{h}'(0)=\frac{(a-1)(1-\lambda)-2\lambda}{1-\lambda}$
and $\tilde{h}'(\pi)=-\frac{(a-1)(1+\lambda)+2\lambda}{1+\lambda}$
we compute 
\begin{align*}
\int_{0}^{\pi}e^{in\tilde{h}(\varphi)}\d\varphi & =\Gamma(1/2)\left(2^{1/2}\left(\tilde{h}''(\varphi_{+})\right)^{-1/2}\right)e^{i\frac{\pi}{4}}n^{-1/2}e^{in\tilde{h}(\varphi_{+})}+\cO\left(n^{-3/2}\right)\\
 & =\frac{\sqrt{2}\Gamma(1/2)e^{in\tilde{h}(\varphi_{+})+i\frac{\pi}{4}}}{\sqrt{k/n}\left[\left(k/n-\alpha_{0}\right)\left(\alpha_{0}^{-1}-k/n\right)\right]^{1/4}}+\cO\left(n^{-3/2}\right)\\
 & =\frac{\sqrt{2\pi}}{\sqrt{k/n}\left[\left(k/n-\alpha_{0}\right)\left(\alpha_{0}^{-1}-k/n\right)\right]^{1/4}}e^{in\tilde{h}(\varphi_{+})+i\frac{\pi}{4}}+\cO\left(n^{-3/2}\right).\\
\end{align*}
We conclude that 
\begin{align*}
 & \frac{1}{\pi}\Re\left\{ \int_{0}^{\pi}e^{in\tilde{h}(\varphi)}\d\varphi\right\} \\
 & =\sqrt{\frac{2}{\pi n}}\frac{\cos\left(n\tilde{h}(\varphi_{+})+\frac{\pi}{4}\right)}{\sqrt{k/n}\left[\left(k/n-\alpha_{0}\right)\left(\alpha_{0}^{-1}-k/n\right)\right]^{1/4}}\big(1+O(n^{-1})\big),
\end{align*}
where $\cO\left({n^{-3/2}}\right)$ is uniform over $k/n\in[\beta,\beta^{-1}]$. 
\end{proof}

\section{Strongly annular functions with small Taylor coefficients}

\label{Annular}

Let us recall that a function $f$ analytic in the unit disc is said
to be strongly annular (we use the notation $f\in\mathcal{S}\mathcal{A}$)
if 
\[
\limsup_{r\to1}\min_{\partial\mathcal{D}(0,r)}|f|=\infty.
\]
The question we are interested in here is how small and how (non)-lacunar
could be the Taylor coefficients $\widehat{f}(n)$ of $f$: 
\[
f(z)=\sum_{n\ge0}\widehat{f}(n)z^{n},\qquad z\in\mathbb{D}.
\]
In 1977, Bonar, Carroll, and Piranian \cite{BCP} produced $f\in\mathcal{S}\mathcal{A}$
such that $\widehat{f}\in c_{0}$. It is clear that if $f\in\mathcal{S}\mathcal{A}$,
then $\widehat{f}\not\in\ell^{2}$. Furthermore, the function constructed
in \cite{BCP} is far from being lacunary. Given $0<p<\infty$, set
\[
\widetilde{\ell}^{p}=\Bigl\{\{a_{n}\}_{n\ge0}:\sum_{k\ge0}\min(|a_{2k}|^{p},|a_{2k+1}|^{p})<\infty\Bigr\}.
\]
Then, the function $f$ constructed in \cite{BCP} is such that $\widehat{f}\in c_{0}\setminus\widetilde{\ell}^{2}$.

In this section we are going to get new results in this direction. 
\begin{thm}
\label{thmlp} Let $2\le p<q$. There exists $f\in\mathcal{S}\mathcal{A}$
such that $\widehat{f}\in\ell^{q}\setminus\widetilde{\ell}^{p}$. 
\end{thm}

Given a positive function $\varphi$ on $\mathbb{R}_{+}$, we set
\[
\ell_{\varphi}^{2}=\Bigl\{\{a_{n}\}_{n\ge0}:\sum_{n\ge0}\frac{|a_{n}|^{2}}{\varphi(1/|a_{n}|)}<\infty\Bigr\}.
\]

\begin{thm}
\label{thmlphi} Let $\varphi$ be an increasing positive function
on $\mathbb{R}_{+}$ such that $\lim_{x\to\infty}\varphi(x)=\infty$.
There exists $f\in\mathcal{S}\mathcal{A}$ such that $\widehat{f}\in\ell_{\varphi}^{2}\setminus\widetilde{\ell}^{2}$. 
\end{thm}

Given $N\ge1$ we denote 
\[
g_{N}(z)=b_{1/2}^{N}(z)=\biggl(\frac{z-\frac{1}{2}}{1-\frac{z}{2}}\biggr)^{N}.
\]
Set 
\[
u_{p}(N)=\begin{cases}
N^{\frac{1}{p}-\frac{1}{2}},\qquad2\le p<4,\\
(\log N)^{\frac{1}{4}}N^{-\frac{1}{4}},\qquad p=4,\\
N^{\frac{1}{3p}-\frac{1}{3}},\qquad p>4,
\end{cases}
\]
and 
\[
v_{p}=\begin{cases}
\frac{1}{2}-\frac{1}{p},\qquad2\le p<4,\\
\frac{1}{3}-\frac{1}{3p},\qquad p>4.
\end{cases}
\]

We use the following corollary of Theorem \ref{Th_Regions_I_VII}
and Theorem \ref{Th:Regions_IV_V_VI}. 
\begin{lem}
\label{lem1} Given $N\ge10$, for some $\delta>0$ we have 
\begin{enumerate}
\item[(i)] $\|g_{N}\|_{H^{\infty}(\mathbb{D})}=1$, 
\item[(ii)] $\min_{\partial\mathcal{D}(0,1-N^{-1})}|g_{N}|\ge e^{-4}$, 
\item[(iii)] $|\widehat{g_{N}}(k)|\le e^{-\delta k}$, $k\ge4N$, 
\item[(iv)] $\|\widehat{g_{N}}\|_{\infty}\lesssim N^{-1/2}$, 
\item[(v)] $\bigl(\sum_{k\ge0}\min(|\widehat{g_{N}}(2k)|^{p},|\widehat{g_{N}}(2k+1)|^{p})\bigr)^{1/p}\asymp\|\widehat{g_{N}}\|_{p}\asymp u_{p}(N)$,
$\qquad p\geq2.$ 
\end{enumerate}
\end{lem}

\begin{proof}
The properties (i) and (ii) follow immediately from the definition
of $g_{N}$. Furthermore, we use that by Theorem \ref{Th_Regions_I_VII}
and Theorem \ref{Th:Regions_IV_V_VI}, we have several upper estimates
on $|\widehat{g_{N}}(k)|$ for different values of $k$. 
\begin{align}
|\widehat{g_{N}}(k)| & \lesssim e^{-cN},\qquad0\le k<\frac{N}{4},\label{eqa}\\
|\widehat{g_{N}}(k)| & \lesssim\frac{\exp(-cN(\frac{1}{3}-\frac{k}{N})^{3/2})}{N^{1/2}(\frac{1}{3}-\frac{k}{N}+N^{-2/3})^{1/4}},\qquad\frac{N}{4}\le k<\frac{N}{3},\label{eqb}\\
|\widehat{g_{N}}(k)| & \lesssim\frac{1}{N^{1/2}(\frac{k}{N}-\frac{1}{3}+N^{-2/3})^{1/4}},\qquad\frac{N}{3}\le k<N,\label{eqc}\\
|\widehat{g_{N}}(k)| & \lesssim\frac{1}{N^{1/2}(3-\frac{k}{N}+N^{-2/3})^{1/4}},\qquad N\le k<3N,\label{eqd}\\
|\widehat{g_{N}}(k)| & \lesssim\frac{\exp(-cN(\frac{k}{N}-3)^{3/2})}{N^{1/2}(\frac{k}{N}-3+N^{-2/3})^{1/4}},\qquad3N\le k<4N,\label{eqe}\\
|\widehat{g_{N}}(k)| & \lesssim e^{-ck},\qquad k\ge4N.\label{eqf}
\end{align}
Next, by Theorem \ref{Th:Regions_IV_V_VI} we have two lower estimates
on $|\widehat{g_{N}}|$ for some intervals of values of $k$: 
\begin{align}
|\widehat{g_{N}}(k)| & \gtrsim N^{-1/3},\qquad\frac{N}{3}\le k<\frac{N}{3}+N^{1/3},\label{eqg}\\
\intertext{and}|\widehat{g_{N}}(k)| & \asymp N^{-1/2}\cos A_{N}(k),\qquad N\le k\le\frac{6N}{5},
\nonumber 
\end{align}
where 
\begin{align*}
A_{N}(t) & =NH_{N}(t)-\frac{\pi}{4},\\
H_{N}(t) & =-\frac{t\varphi_{N}(t)}{N}+\psi(\varphi_{N}(t)),\\
\psi'(s) & =\frac{3}{5-4\cos s},\\
\varphi_{N}(t) & \in(0,\pi),\\
\cos\varphi_{N}(t) & =\frac{5}{4}-\frac{3N}{4t}.
\end{align*}
Furthermore, 
\begin{align*}
A'_{N}(t) & =-\varphi_{N}(t),\\
A''_{N}(t) & =\frac{3N}{4t^{2}\sin\varphi_{N}(t)}.
\end{align*}
For $t\in[N,6N/5]$ we have 
\begin{align*}
\cos\varphi_{N}(t) & \in[1/2,5/8],\\
A''_{N}(t) & \asymp1/t,\\
-\pi/3 & \le A'_{N}(t)\le-\pi/4,
\end{align*}
and, hence, 
\[
-\frac{2\pi}{5}\le A_{N}(k+1)-A_{N}(k)\le-\frac{\pi}{5},\qquad N\le k\le\frac{6N}{5}-1.
\]
Thus, for every $k\in[N,\frac{6N}{5}-1]$, 
\begin{equation}
\min(|\widehat{g_{N}}(k)|,|\widehat{g_{N}}(k+1)|)\gtrsim N^{-1/2}.\label{eqi}
\end{equation}
Finally, (iii) is \eqref{eqf}, (iv) follows from \eqref{eqa}--\eqref{eqf},
and (v) follows from \eqref{eqa}--\eqref{eqi}. 
\end{proof}
Another proof of the second asymptotic relation in Lemma~\ref{lem1}~(v)
is given in \cite{SzZa1}. 
\begin{proof}[Proof of Theorem~\ref{thmlp}]
Choose $r\in(p,q)\setminus\{4\}$. Given an integer $A>1$, set 
\[
f(z)=\sum_{k\ge1}A^{kv_{r}}g_{A^{k}}(z)z^{A^{k}}.
\]
First of all, the function $f$ is analytic in the unit disc. Furthermore,
\begin{multline*}
\min_{\partial\mathcal{D}(0,1-A^{-k})}|f|\ge\min_{\partial\mathcal{D}(0,1-A^{-k})}|A^{kv_{r}}g_{A^{k}}(z)z^{A^{k}}|-\sum_{s\ge1,\,s\not=k}\max_{\partial\mathcal{D}(0,1-A^{-k})}|A^{sv_{r}}g_{A^{s}}(z)z^{A^{s}}|\\
\ge e^{-6}A^{kv_{r}}-\sum_{1\le s<k}A^{sv_{r}}-\sum_{s>k}A^{sv_{r}}\exp(-A^{s-k})\gtrsim A^{kv_{r}}\to\infty,\qquad k\to\infty,
\end{multline*}
if $A^{v_{r}}\ge A_{0}$. Thus, $f\in\mathcal{S}\mathcal{A}$.

Next, given $\eta>2$ we have 
\begin{multline*}
\sum_{n\ge0}|\widehat{f}(n)|^{\eta}=\sum_{\ell\ge1}\sum_{A^{\ell}\le n<A^{\ell+1}}|\widehat{f}(n)|^{\eta}\\
=\sum_{\ell\ge1}\Bigl(\sum_{A^{\ell}\le n<A^{\ell+1}}A^{\eta\ell v_{r}}|\widehat{g_{A^{\ell}}}(n-A^{\ell})|^{\eta}+O(A^{\ell+1}(\ell A^{\eta\ell v_{r}}e^{-\delta A^{\ell}})^{\eta})\Bigr)\\
=O(1)+\sum_{\ell\ge1}A^{\eta\ell v_{r}}\sum_{n\ge0}|\widehat{g_{A^{\ell}}}(n)|^{\eta},
\end{multline*}
if $A\ge A_{1}(\delta)$. Thus, $\widehat{f}\in\ell^{q}$ and $\widehat{f}\notin\widetilde{\ell}^{p}$. 
\end{proof}
\begin{proof}[Proof of Theorem~\ref{thmlphi}]
Given $A>1$, choose integer $N_{k}$ such that 
\[
N_{k+1}\ge AN_{k},\qquad\min_{[N_{k}^{1/4},\infty)}\varphi\ge A^{3k},\qquad k\ge1.
\]
Now set 
\[
f(z)=\sum_{k\ge1}A^{k}g_{N_{k}}(z)z^{N_{k}}.
\]
As in the proof of Theorem~\ref{thmlp}, $f$ is analytic in the
unit disc and for $A\ge A_{0}$ we have 
\[
\min_{\partial\mathcal{D}(0,1-N_{k}^{-1})}|f|\gtrsim A^{k}.
\]
Thus, $f\in\mathcal{S}\mathcal{A}$.

Next, 
\[
\sum_{n\ge0}\frac{|\widehat{f}(n)|^{2}}{\varphi(1/|\widehat{f}(n)|)}\lesssim O(1)+\sum_{k\ge1}A^{2k}\frac{\|g_{N_{k}}\|_{2}^{2}}{\varphi(cN_{k}^{1/3})}<\infty,
\]
and, again by Lemma~\ref{lem1}, we conclude that $\widehat{f}\in\ell_{\varphi}^{2}$
and $\widehat{f}\notin\widetilde{\ell}^{2}$. 
\end{proof}

\subsection{Flat polynomials}

Here we discuss an alternative approach to Theorems~\ref{thmlp}
and \ref{thmlphi} in such a way that they use different constructions
of flat polynomials. 
\begin{lem}
\label{lem2} Given a large $N$, there exists a polynomial $g_{N}$
of degree $N$ such that 
\begin{enumerate}
\item[(i)] $\|g_{N}\|_{H^{\infty}(\mathbb{D})}=1$, 
\item[(ii)] $\min_{\partial\mathcal{D}(0,1-N^{-2})}|g_{N}|\gtrsim1$, 
\item[(iii)] $\|\widehat{g_{N}}\|_{\infty}\lesssim N^{-1/2}$, 
\item[(iv)] $\bigl(\sum_{0\le k\le N}\min(|\widehat{g_{N}}(2k)|^{p},|\widehat{g_{N}}(2k+1)|^{p})\bigr)^{1/p}\asymp\|\widehat{g_{N}}\|_{p}\asymp N^{\frac{1}{p}-\frac{1}{2}}$. 
\end{enumerate}
\end{lem}

One can easily modify the proofs of Theorems~\ref{thmlp} and \ref{thmlphi} in such a way that 
they use Lemma~\ref{lem2} instead of Lemma~\ref{lem1}.

Furthermore, Lemma~\ref{lem2} follows from a 1978 result of Körner.
Solving a Littlewood problem he established in \cite[Theorem 6]{Kor}
the existence of polynomials of degree $N$ with unimodular coefficients
equivalent to $\sqrt{N}$ on the unit circle. This gives Lemma~\ref{lem2}
immediately. This result of Körner is non-constructive. For further
progress in this direction including some explicit constructions see
\cite{BoBo} and the recent paper \cite{BBMST}.


\begin{thebibliography}{10}
\bibitem{AlDe1} Ali Mehmeti F., Dewez F., \textit{Explicit error
estimates for the stationary phase method I: The influence of amplitudes
singularities}, arXiv:1412.5789v1 (2014).

\bibitem{AlDe2} Ali Mehmeti F., Dewez F., \textit{Explicit error
estimates for the stationary phase method II: Interaction of amplitude
singularities with stationary points}, arXiv:1412.5792v1 (2014).

\bibitem{And2} Andersson J., \textit{On some power sum problems of
Turán and Erdös}, Acta Math.\ Hungar.\ \textbf{70} (1996) 305--316.

\bibitem{And1} Andersson J., \textit{Turán's problem 10 revisited},
arXiv:math/0609271, (2008).

\bibitem{BBMST} Balister, P., Bollobás, B., Morris, R., Sahasrabudhe,
J., Tiba, M., \textit{Flat Littlewood polynomials exist}, Ann.\ Math.\ \textbf{192}
(2020) 977--1004.

\bibitem{BeHe} Beurling A., Helson H., \textit{Fourier--Stieltjes
transforms with bounded powers}, Math.\ Scand.\ \textbf{1} (1953)
120--126.

\bibitem{BeBo} Bernal-González, L., Bonilla, A., \textit{Families
of strongly annular functions: linear structure}, Rev.\ Mat.\ Complut.\ \textbf{26}
(2013) 283--297.

\bibitem{BlHa} Bleistein N., Handelsman R. A., \textit{Asymptotic
Expansions of Integrals}, Dover Publications, Inc., New York, second
edition, 1986.

\bibitem{BlSh} Blyudze M.~Y., Shimorin S.~M., \textit{Estimates
of the norms of powers of functions in certain Banach spaces}, J.
Math.\ Sci.\ \textbf{80} (1996) 1880--1891.

\bibitem{BoBo} Bombieri, E., Bourgain, J., \textit{On Kahane's ultraflat
polynomials}, J. Eur.\ Math.\ Soc.\ \textbf{11} (2009) 627--703.

\bibitem{Bon} Bonar D. D, \textit{On Annular Functions}, VEB Deutscher
Verlag der Wissenschaften, Berlin, 1971.

\bibitem{BCE} Bonar D. D, Carroll F. W., Erdös, P., \textit{Strongly
annular functions with small coefficients and related results}, Proc.\ Am.\ Math.\ Soc.\ \textbf{67}
(1977) 129--132.

\bibitem{BCP} Bonar D. D, Carroll F. W., Piranian G., \textit{Strongly
annular functions with small Taylor coefficients}, Math.\ Z. \textbf{156}
(1977) 85--91.

\bibitem{Bor} Borovikov V.A., \textit{Uniform Stationary Phase Method},
Institute of Engineering and Technology, London (1994).

\bibitem{Bru} de Bruijn N. G., \textit{Asymptotic Methods in Analysis},
North--Holland, Amsterdam, 1958.

\bibitem{CFU} Chester C., Friedman B., Ursell F., \textit{An extension
of the method of steepest descents}, Mathematical Proceedings of the
Cambridge Philosophical Society \textbf{53} (1957) 599--611.

\bibitem{Cop} Copson E. T., \textit{Asymptotic Expansions}, Cambridge
Tracts in Mathematics, 1965.

\bibitem{Daq1} Daquila, R., \textit{Strongly annular solutions of
Mahler's functional equation}, Complex Var.\ Theory Appl.\ \textbf{32}
(1997) 99--104.

\bibitem{Daq2} Daquila, R., \textit{Approximations by strongly annular
solutions of functional equations}, Proc.\ Am.\ Math.\ Soc.\ \textbf{138}
(2010) 2505--2511.

\bibitem{Erd} Erdélyi E., \textit{Asymptotic Representations of Fourier
Integrals and The Method of Stationary Phase}, J. Soc.\ Indust.\ Appl.\ Math.\ \textbf{3}
(1955) 17--27.

\bibitem{ErRe} Erdös P., Rényi A., \textit{A probabilistic approach
to problems of Diophantine approximation}, lllinois J. Math.\ \textbf{1}
(1957) 303--315.

\bibitem{Fed2} Fedoryuk M. V., \textit{The stationary phase methods
and pseudo-differential operators}, Russian Math.\ Surveys \textbf{26}
(1971) 65--115.

\bibitem{Fed1} Fedoryuk M. V., \textit{Metod Perevala (Saddle-Point
Method)}, Nauka, Moscow, 1977.

\bibitem{Gir} Girard D. M., \textit{The behavior of the norm of an
automorphism of the unit disk}, Pacific J. Math. \textbf{47} (1973)
443--456.

\bibitem{GMP} Gluskin E., Meyer M., Pajor A., \textit{Zeros of analytic
functions and norms of inverse matrices}, Israel J. Math.\ \textbf{87}
(1994) 225--242.

\bibitem{How} Howell, R., \textit{Annular functions in probability},
Proc.\ Am.\ Math.\ Soc.\ \textbf{52} (1975) 217--221.

\bibitem{Kah} Kahane J.-P., \textit{Sur certaines classes de series
de Fourier absolument convergentes}, J. Math.\ Pure Appl.\ \textbf{9}
(1956) 249--259.

\bibitem{Kah1} Kahane, J.-P., \textit{Some random series of functions},
Cambridge Studies in Advanced Mathematics, 5. Cambridge University
Press, Cambridge, 1985.

\bibitem{Kor} Körner T.~W., \textit{On a polynomial of Byrnes},
Bull.\ London Math.\ Soc.\ \textbf{12} (1980) 219--224.

\bibitem{LLQRP} Lefèvre P., Li D., Queffélec H., Rodr{í}guez-Piazza
L., \textit{Boundedness of composition operators on general weighted
Hardy spaces of analytic functions}, arXiv:2011.14928 (2020).

\bibitem{Leib} Leibenson Z.~L., \textit{On the ring of functions
with absolutely convergent Fourier series}, Uspekhi Mat.\ Nauk \textbf{9}
(1954) 3, 157--162.

\bibitem{Mon} Montgomery H. L., Ten lectures on the interface between
analytic number theory and harmonic analysis, AMS, 1994.

\bibitem{Nik2} Nikolski N., \textit{Operators, Function, and Systems:
An easy reading, Vol.1}, Amer. Math. Soc. Monographs and Surveys,
2002.

\bibitem{Nik1} Nikolski N., \textit{Condition Numbers of Large Matrices
and Analytic Capacities}, St.\ Petersburg Math.\ J. \textbf{17}
(2006) 641--682.

\bibitem{Que2} Queffélec H., \textit{Sur un théorème de Gluskin--Meyer--Pajor},
C. R. Acad.\ Sci.\ Paris \textbf{317} (1993) 155--158.

\bibitem{Que1} Queffélec H., \textit{Norm of the inverse of a matrix;
solution to a problem of Schäffer}, Harmonic Analysis from the Pichorides
viewpoint, Publ. Math. Orsay, 96-01, 68--87 (1996).

\bibitem{Red} Redett, D., \textit{Strongly annular functions in Bergman
space}, Comput. Methods Funct.\ Theory \textbf{7} (2007) 429--432.

\bibitem{Rud} Rudin W., \textit{Fourier Analysis on Groups}, New
York, Interscience, 1962.

\bibitem{Sch} Schäffer J. J., \textit{Norms and determinants of linear
mappings}, Math.\ Z. \textbf{118} (1970) 331--339.

\bibitem{SzZa3} Szehr O., Zarouf R., \textit{On the asymptotic behavior
of Jacobi polynomials with varying parameters}, arXiv:1605.02509 (2016).

\bibitem{SzZa2} Szehr O., Zarouf R., \textit{A constructive approach
to Schäffer's conjecture}, arXiv:1705.10704 (2017).

\bibitem{SzZa1} Szehr O., Zarouf R., \textit{$l_{p}$-norms of Fourier
coefficients of powers of a Blaschke factor}, J. Anal.\ Math.\ \textbf{140}
(2020) 1--30.

\bibitem{SzZa4} Szehr O., Zarouf R., \textit{Explicit counterexamples
to Schäffer's conjecture}, J. Math.\ Pures et Appl., to appear.

\bibitem{Tem} Temme N. M., \textit{Asymptotic methods for integrals},
World Scientific Publishing Co., 2015.

\bibitem{Tur} Turán P., On a new method of analysis and its applications,
Pure and Applied Mathematics, New-York, (1984.

\bibitem{Won} Wong R., \textit{Asymptotic Approximations of Integrals},
Society for Industrial and Applied Mathematics, 2001. 
\end{thebibliography}
\end{document}